%% file: main.tex
\documentclass[article, onesided, a4paper]{amsart}

\input{preamble}

\title{Taylor Morphisms}
\author{Gabriel Ng}
\address{University of Manchester, Oxford Road, Manchester M13 9PL, United Kingdom}
\curraddr{Universit\'e de Mons, Place du Parc 20, 7000 Mons, Belgium}
\email{gabriel.ng@umons.ac.be}
\thanks{The research for this article was conducted whilst the author was supported by the University of Manchester through a President's Doctoral Scholarship and the PGR Incentive Fund.}
\date{6th March 2026}

\begin{document}

\begin{abstract}
    We study generalised Taylor morphisms, functors which construct differential ring homomorphisms from ring homomorphisms in a uniform way, analogous to the Taylor expansion for smooth functions. We generalise the construction of the twisted Taylor morphism by Le\'on S\'anchez and Tressl to arbitrary differential rings by `twisting' the ring of Hurwitz series, and prove that this results in a functor which is the right adjoint to a certain forgetful functor. We therefore give a concrete characterisation of all generalised Taylor morphisms over all differential rings with finitely many commuting derivations.
\end{abstract}

\maketitle

\section{Introduction}

In differential algebra, the classical Taylor morphism (Example \ref{classical_tm_eg}) is a natural algebraic analogue of the Taylor series expansion of a smooth function at a point, i.e. for a differential ring $(A, \bdelta)$ (in $m$ commuting derivations), and ring homomorphism $\phi: A \to K$, where $K$ is a $\Q$-algebra, the Taylor morphism of $\phi$  is the differential ring homomorphism
\begin{align}
    T_\phi: (A, \bdelta) &\to \left( K[[\bt]], \ddbt \right) \\
    a &\mapsto \sum_{\alpha \in \N^m} \frac{\phi(\bdelta^\alpha a)}{\alpha!} \bt^\alpha.
\end{align}
 This allows us to construct differential ring homomorphisms into the ring of formal power series $K[[\bt]]$ from non-differential ring homomorphisms into the ring of coefficients $K$ in a uniform way.

This construction also has a certain categorical interpretation: the functor which sends a $\Q$-algebra $K$ to the differential ring of formal power series $(K[[\bt]], \ddbt)$ is the right adjoint to the forgetful functor from the category of differential $\Q$-algebras to the category of $\Q$-algebras, and the natural bijection of hom-sets is given precisely by the classical Taylor morphism.


In \cite{Keigher1975} and \cite{Keigher1997}, Keigher studies the properties of the \emph{differential ring of Hurwitz series} $(H(K), \bd_K)$ over an arbitrary commutative ring $K$ of any characteristic, along with the functor $H$ which sends a ring $K$ to the ring of Hurwitz series over $K$. Associated to this is a construction which we call the \emph{`Hurwitz morphism'}, which is analogous to the classical Taylor morphism in the case where $K$ is not necessarily a $\Q$-algebra (Example \ref{hurwitz_morphism_eg}). The functor $H$ is the right adjoint to the forgetful functor $U$ from the category of differential rings to the category of rings, and its restriction to the subcategory of $\Q$-algebras is precisely the formal power series construction above.

In \cite{SanchezTressl2020}, Le\'on S\'anchez and Tressl construct the \emph{twisted Taylor morphism} as a tool to study differentially large fields (Example \ref{ttm_eg}). This is a modification of the classical Taylor morphism in the following sense: suppose that $(K, \bpartial)$ is a differential $\Q$-algebra, and $\phi: A \to K$ is a ring homomorphism. Then, the twisted Taylor morphism of $\phi$ is a differential ring homomorphism
\[
T^*_\phi: (A, \bdelta) \to \left(K[[\bt]], \bpartial + \ddbt \right)
\]
where $\bpartial$ is the derivation on $K[[\bt]]$ which acts by $\bpartial$ on coefficients. Observe that in this case, $(K[[\bt]], \bpartial + \ddbt)$ is a differential $(K, \bpartial)$-algebra with structure map given by the natural inclusion. Further, the twisted Taylor morphism preserves $K$-algebra homomorphisms, i.e. if $A$ is a differential $(K, \bpartial)$-algebra, and $\phi$ is a $K$-algebra homomorphism, then $T^*_\phi$ is a differential $(K, \partial)$-algebra homomorphism. The main application of this in \cite{SanchezTressl2020} is to prove various equivalent characterisations of differential largeness in terms of existential closure conditions.

In this article, we show that all of these constructions are special cases of functors which we will call (generalised) \emph{Taylor morphisms} (Definiton \ref{taylor_morphism_defn}). Formally, for a differential ring $K$ and a differential $K$-algebra $L$, a \textbf{$K$-Taylor morphism} for $L$ is a map $T$ which sends pairs $(A, \phi)$, where $A$ is a differential ring, and $\phi: A \to K$ is a ring homomorphism, to a differential ring homomorphism $T_\phi: A \to L$, and satisfies the following axioms:
\begin{itemize}
    \item[(TM1)] If $\phi: A \to K$ is differential, then $T_\phi = \eta_L \circ \phi$, where $\eta_L: K \to L$ is the structure map of $L$ as a $K$-algebra;
    \item[(TM2)] If $A, B$ are differential rings, $\phi: A \to K$ and $\psi: B \to K$ are ring homomorphisms and $\chi: A \to B$ is a differential ring homomorphism such that $\phi = \psi \circ \chi$, then $T_\phi = T_\psi \circ \chi$.
\end{itemize}
 It is not difficult to see that the condition (TM2) implies that a $K$-Taylor morphism is a functor between two comma categories (Lemma \ref{k_tm_is_functor_lem}). A consequence of this axiomatisation is that $K$-Taylor morphisms preserve $K$-algebra homomorphisms, i.e. if $A$ is a differential $K$-algebra, and $\phi: A \to K$ is a $K$-algebra homomorphism, then $T_\phi: K \to L$ is a differential $K$-algebra homomorphism (Proposition \ref{k_tm_preserves_k_algs_prop}).

By equipping a $\Q$-algebra $K$ with a tuple of trivial derivations $\bzero$, the classical Taylor morphism can be viewed as a $K$-Taylor morphism for $(K[[\bt]], \ddbt)$; similarly, when $K$ is an arbitrary ring with trivial derivations, the Hurwitz morphism is a $K$-Taylor morphism for $(H(K), \bd_K)$ where $K$ is an arbitrary ring; and finally, the twisted Taylor morphism is a $K$-Taylor morphism for $(K[[\bt]], \bdelta + \ddbt)$, where $K$ is a differential $\Q$-algebra equipped with a tuple $\bdelta$ of possibly non-trivial derivations.

In Section \ref{restrictions_section}, we show that every $K$-Taylor morphism is completely determined by its restriction to differentially finitely generated $K$-algebras (Proposition \ref{ftm_implies_tm}). This allows us, in Section \ref{utm_section}, to construct a (small) diagram in the category of differential $K$-algebras whose cocones have a bijective correspondence with $K$-Taylor morphisms (Proposition \ref{taylor_morphism_is_cone}). From this, we show in Theorem \ref{universal_TM_construction_thm} that every generalised Taylor morphism is the image of a \emph{universal} Taylor morphism, which corresponds with the colimit of this diagram.

In Section \ref{eval_maps_section}, we consider the notion of an \emph{evaluation map} for a $K$-Taylor morphism $T$ for $L$ (Definition \ref{eval_maps_defn}), which is a generalisation of the `evaluation at 0' map $\ev_0: K[[\bt]] \to K$ for the twisted and classical Taylor morphisms. An evaluation map $\ev: L \to K$ for $T$ (if one exists), is in some sense the inverse to the Taylor morphism $T$ (Lemma \ref{ev_commutes_tm}).
In Proposition \ref{utm_admits_eval_map_prop}, we construct an evaluation map for the universal Taylor morphism, and as a consequence we prove that the universal Taylor morphism is characterised by the existence of such a map (Theorem \ref{universal_tm_char_by_ev_thm}).
This shows that the twisted Taylor morphism and Hurwitz morphism are the universal Taylor morphism for a differential $\Q$-algebra, and arbitrary constant ring, respectively (Examples \ref{eval_map_egs}). Another consequence is that the ring of differentially algebraic power series over $K$ does not admit a $K$-Taylor morphism (Proposition \ref{diff_alg_power_series_no_tm_prop}).

We prove also that the universal Taylor morphism naturally results in a functor $F$ from the category $\DRing_\Ring$, whose objects are differential rings and morphisms are ring homomorphisms, to the category of differential rings $\DRing$ (Lemma \ref{UTM_is_functor_lemma}). 
This functor is the right adjoint to the inclusion $\DRing \to \DRing_\Ring$ (Theorem \ref{utm_functor_is_right_adjoint_thm}), and the natural bijection of hom-sets is given by the universal Taylor morphism with inverse given by compositon with the evaluation map.

In the final section, we give a concrete characterisation of the universal Taylor morphism, by generalising the construction of the twisted Taylor morphism by Le\'on S\'anchez and Tressl to rings of arbitrary characteristic in finitely many commuting derivations. We achieve this by constructing a twisting for the Hurwitz morphism (Corollary \ref{hurwitz_twisting_cor}). In Theorem \ref{twisted_hurwitz_is_utm_thm}, we show that the twisted Hurwitz morphism coincides with the universal Taylor morphism constructed in Section \ref{utm_section} by explicitly exhibiting an evaluation map. This gives a complete explicit characterisation of all $K$-Taylor morphisms for any differential ring $K$.

Throughout this article, all rings will be assumed to be commutative and unital. Rings and fields may be of arbitrary characteristic, unless otherwise specified.

\section{Differential Algebraic Preliminaries}

Let $K$ be a ring. Recall that a \textbf{derivation} on $K$ is an additive map $\delta: K \to K$ satisfying the product rule, i.e. for any $a, b \in K$,
\[
\delta(ab) = \delta(a)b + a \delta(b).
\]
A \textbf{differential ring} is a ring $K$ equipped with a tuple $\bdelta= (\delta_1,...,\delta_m)$ of $m$ \textbf{commuting} derivations, i.e. for any $1 \leq i \leq j \leq m$, $\delta_i \circ \delta_j = \delta_j \circ \delta_i$. A \textbf{differential subring} of $(K, \bdelta)$ is a subring $L$ of $K$ such that $L$ is closed under $\delta_i$ for each $i$. A \textbf{differential field} is a differential ring which is a field.
If no confusion arises, we will sometimes drop the reference to the derivation of a differential field, e.g. write $K$ for $(K, \bdelta)$.

A \textbf{differential ring homomorphism} $(K, \bdelta) \to (L, \bpartial)$ is a ring homomorphism $\phi: K \to L$ such that for any $a \in K$ and $i = 1,...,m$, $\phi(\delta_i(a)) = \partial_i(\phi(a))$.

The \textbf{constant subring} $C_K$ of the differential ring $(K, \bdelta)$ is the common kernel of $\bdelta$, i.e.
\[
C_K = \{a \in K: \delta_i (a) = 0  \text{ for each }i\}.
\]
The \textbf{trivial derivation} on a ring $K$ is the constant function $0: K \to K$ which maps every element to $0$. Write $\bzero$ for the $m$-tuple of trivial derivations. We say that a differential ring $(K, \bdelta)$ is \textbf{constant} if $\bdelta = \bzero$, i.e. $C_K = K$.

For a differential ring $(K, \bdelta)$, a \textbf{differential $K$-algebra} is a differential ring $(A, \bpartial)$ equipped with a differential ring homomorphism $\eta_A: K \to A$ known as the \textbf{structure map}. For differential $K$-algebras $A, B$, a \textbf{differential $K$-algebra homomorphism} $\phi: A \to B$ is a $K$-algebra homomorphism which is also a differential ring homomorphism.

Borrowing terminology from geometry, if $K$ is a (differential) ring and $A$ a (differential) $K$-algebra, a \textbf{(differential) $K$-rational point} (or simply \textbf{(differential) $K$-point}) is a (differential) $K$-algebra homomorphism $\phi: A \to K$.

For a differential ring $(A, \bdelta)$, a \textbf{differential ideal} of $A$ is an ideal $I \subseteq A$ such that $I$ is closed under $\bdelta$. 
That is, for any $i$ and any $a \in I$, $\delta_i (a)  \in I$.
A differential ideal $I \subseteq A$ is said to be \textbf{prime} (respectively, \textbf{maximal}) if $I$ is prime (respectively, maximal) as an ideal.
The quotient $A/I$ is naturally equipped with derivations inherited from $A$, i.e. $\delta_i(a + I) = \delta_i(a) + I$. 
The kernel of a differential ring homomorphism is a differential ideal.

\begin{nota}

    For convenience, we will also use multi-index notation. A \textbf{multi-index} (of length $m$) is a tuple $\alpha = (\alpha_1,...,\alpha_m) \in \N^m$. We write the following:
    \begin{itemize}
        \item For $\alpha, \beta \in \N^m$, $\alpha \leq \beta$ if $\alpha_i \leq \beta_i$ for each $i$.
        \item For $\alpha, \beta \in \N^m$, $\alpha + \beta = (\alpha_1+\beta_1,...,\alpha_m + \beta_m)$. Similarly, if $\alpha \geq \beta$, $\alpha - \beta \coloneqq (\alpha_1 - \beta_1,...,\alpha_m - \beta_m)$.
        \item For $\alpha \in \N^m$, $\alpha! \coloneqq \alpha_1!...\alpha_m!$. For $\beta \leq \alpha \in \N^m$, $\binom{\alpha}{\beta} = \frac{\alpha!}{\beta! (\alpha - \beta)!} = \prod_i \binom{\alpha_i}{\beta_i}$
        \item If $\bt = (t_1,...,t_m)$ is a tuple of indeterminates of length $m$, and $\alpha \in \N^m$, $\bt^\alpha \coloneqq t_1^{\alpha_1}...t_m^{\alpha_m}$
        \item For $\bdelta = (\delta_1,...,\delta_m)$ a tuple of derivations, $\bdelta^\alpha \coloneqq \delta_1^{\alpha_1}...\delta_m^{\alpha_m}$.
        \item If $a$ is an element of a ring $K$ and $\alpha \in \N^m$, $a^\alpha \coloneqq a^{|\alpha|}$, where $|a| = \sum_{i=1}^m a_i$.
    \end{itemize}
\end{nota}

The \textbf{ring of differential polynomials} over a differential ring $(K, \bdelta)$ in indeterminates $\bar{x} = (x_i)_{i \in I}$ is the polynomial ring
\[
K\{ \bar{x} \} \coloneqq K[\bdelta^\alpha x_i: \alpha \in \N^m, i \in I]
\]
equipped with derivations $\bdelta$ naturally extending $\bdelta$ on $K$, and satisfy $\bdelta^\alpha(\bdelta^\beta x_i) = \bdelta^{\alpha + \beta} x_i$. A differential polynomial $f(\bar{x}) \in K\{\bar{x}\}$ can be evaluated at $\bar{a} \in (L, \bpartial)$, where $(L, \bpartial)$ is a differential $(K, \bdelta)$-algebra, in the natural way. We say that $a \in (L, \bpartial)$ is \textbf{differentially algebraic over $K$} if there is some nonzero $f \in K\{x\}$ such that $f(a) = 0$, and \textbf{differentially transcendental}, otherwise.

A differential $K$-algebra $A$ is said to be \textbf{differentially finitely generated} if there is a surjective $K$-algebra homomorphism $K\{\bar{x}\} \to A$, where $\bar{x}$ is a finite tuple of indeterminates.

\begin{eg}
    Let $(K, \bdelta)$ be a differential ring. Write $\bt = (t_1,...,t_m)$, and write $K[[\bt]] = K[[t_1,...,t_m]]$ for the ring of formal power series in $m$ indeterminates with coefficients in $K$. The elements of $K[[\bt]]$ are formal sums of the form
    \[
    \sum_{\alpha \in \N^m} a_\alpha \bt^\alpha
    \]
    with $a_\alpha \in K$, and addition and multiplication defined coordinatewise and by the usual Cauchy product, respectively.

    The ring $K[[\bt]]$ is naturally equipped with the commuting tuple of formal derivatives
    \[
    \ddbt = \left( \frac{\dd}{\dd t_1},...,\frac{\dd}{\dd t_m} \right).
    \]
\end{eg}

\section{Some Categorical Facts} 
We define a number of commonly occurring categories and constructions, and collect some useful properties of the category of differential $K$-algebras.

\begin{defn}
    We fix the following notation for certain categories and functors:
    \begin{itemize}
        \item The category $\Ring$ is the category of (commutative) rings.
        \item The category $\DRing$ denotes the category of differential rings.
        \item For a fixed differential ring $K$, The category $K\Alg$ is the category of $K$-algebras.
        \item The category $K\DAlg$ is the category of differential $K$-algebras.
        \item The category $K\DAlg_\dfg$ is the category of differentially finitely generated differential $K$-algebras.
        \item The category $\mathbf{1}$ is the category containing a unique object $\bullet$ with the only morphism being the identity $\id: \bullet \to \bullet$.
        \item $U$ denotes a forgetful functor, usually $\DRing \to \Ring$.
        \item For an object $K$ in an arbitrary category $\mathcal{C}$, the functor $1_K: \mathbf{1} \to \mathcal{C}$ sends the unique object in $\mathbf{1}$ to $K$.
        \item For an arbitrary category $\mathcal{C}$, $\id_\mathcal{C}: \mathcal{C} \to \mathcal{C}$ is the identity functor on $\mathcal{C}$.
    \end{itemize}
\end{defn}

We recall a standard construction from category theory, known as \emph{comma categories}, which we will use extensively. We then present a number of useful facts.

\begin{defn}[Comma Category]
    Let $\mathcal{C}, \mathcal{D}, \mathcal{E}$ be categories, and $F: \mathcal{C} \to \mathcal{E}$ and $G: \mathcal{D} \to \mathcal{E}$ be functors. The \textbf{comma category} $(F \downarrow G)$ is the category with objects given by triples of the form $(c, d, \phi)$, where $c \in \mathcal{C}$, $d \in \mathcal{D}$, and $\phi: F(c) \to G(d)$ is a morphism in $\mathcal{E}$. A morphism from $(c_1, d_1, \phi_1)$ to $(c_1, d_1, \phi_1)$ in $(F \downarrow G)$, is a pair $(\alpha, \beta)$, where $\alpha: c_1 \to c_2$ is a morphism in $\mathcal{C}$, $\beta: d_1 \to d_2$ is a morphism  in $\mathcal{D}$, such that the following square commutes:
    \[
    \begin{tikzcd}
        F(c_1) \ar[r, "F(\alpha)"] \ar[d, "\phi_1"]& F(c_2) \ar[d, "\phi_2"]\\
        G(d_1) \ar[r, "G(\beta)"]& G(d_2)
    \end{tikzcd}
    \]
    The composition of morphisms $(\alpha_1, \beta_1): (c_1, d_1, \phi_1) \to (c_2, d_2, \phi_2)$ and $(\alpha_2, \beta_2): (c_2, d_2, \phi_2) \to (c_3, d_3, \phi_3)$ is defined by
    \[
    (\alpha_2, \beta_2) \circ (\alpha_1, \beta_1) = (\alpha_2 \circ \alpha_1, \beta_2 \circ \beta_1).
    \]

    The \textbf{domain functor} $D: (F \downarrow G) \to \mathcal{C}$ is the functor which maps objects by $(c, d, \phi) \mapsto c$ and morphisms by $(\alpha, \beta) \mapsto \alpha$. 
\end{defn}

\begin{eg}
    For an arbitrary category $\mathcal{C}$ and object $K$ in $\mathcal{C}$, the slice category $\mathcal{C}/K$ is (isomorphic to) the comma category $(\id_\mathcal{C} \downarrow 1_K)$.
\end{eg}

\begin{thm}[{\cite[Section 5.2, Theorem 3]{RydeheardBurstall1988}}] \label{certain_comma_category_is_cocomplete_thm}
    Let $\mathcal{C}, \mathcal{D}, \mathcal{E}$ be categories, with $\mathcal{C}$ and $\mathcal{D}$ cocomplete. Let $F: \mathcal{C} \to \mathcal{E}$ and $G: \mathcal{D} \to \mathcal{E}$ be functors with $F$ cocontinuous. Then, the comma category $(F \downarrow G)$ is cocomplete. In particular, the domain functor $D: (F \downarrow G) \to \mathcal{C}$ is cocontinous.
\end{thm}

We now record various properties of the category $K\DAlg$ of differential $K$-algebras. In particular, we show that it is cocomplete, and the forgetful functor $U: K\DAlg \to K\Alg$ is cocontinuous.

\begin{prop} \label{cat_of_diff_K_algs_is_cocomplete_prop}
    Let $K$ be a differential ring. Then the category of differential $K$-algebras is cocomplete.
\end{prop}
\begin{proof}
    By \cite[Corollary V.2.2]{MacLane1978}, it suffices to show that the category of differential $K$-algebras has coequalisers of all pairs of arrows and all small coproducts.

    We first show that for all differential $K$-algebras $(A, \bdelta), (B, \bpartial)$ and pairs of differential $K$-algebra homomorphisms $f, g: A \to B$, there is a differential $K$-algebra $C$ and morphism $h: B \to C$ such that $h\circ f = h \circ g$, and, for any differential $K$-algebra $D$ and morphism $s: B \to D$ such that $s \circ f = s \circ g$, there is a unique morphism $e: C \to D$ such that the following commutes:
    \[
    \begin{tikzcd}
        A \ar[r, shift left, "f"] \ar[r, shift right, "g"'] & B \ar[r, "h"] \ar[rd, "s"'] & C \ar[d, dashed, "e"] \\
        & & D
    \end{tikzcd}
    \]
    This is a generalisation of the analogous result for commutative rings. Let $I$ be the ideal of $B$ generated by all elements of the form $f(a) - g(a)$, where $a \in A$. Then, $I$ is a differential ideal of $B$: as $f$ and $g$ are differential, for any $i$, $\partial_i(f(a)-g(a)) = f(\delta_i(a)) - g(\delta_i(a))$. Now, let $C$ be the differential $K$-algebra $B/I$ and let $h: B \to C$ be the quotient map. Clearly, by construction, $h \circ f = h \circ g$.

    Now, for any element $b+I \in C$, define $e: C \to D$ by setting $e(b+I) = s(b)$. This is well defined as $s(f(a)) = s(g(a))$ for all $a \in A$, and thus $s(b) = 0$ for any $b \in I$. Further, $e$ is differential, as $s$ is differential.
    Uniqueness is clear, as any $K$-algebra homomorphism $\phi: C \to D$ such that the above diagram commutes necessarily satisfies the condition that $\phi(b + I) = s(b)$.

    For the second part, we observe that for pairs of differential $K$-algebras $(A, \bdelta)$, $(B, \bpartial)$, their coproduct is given by the tensor product $A \otimes_K B$, with derivations $\bd$ given by $ \dd_i(a\otimes b) = \delta_i a \otimes b + a \otimes \partial_i b$.
    It is a straightforward computation to see that $\bd$ is a commuting family of derivations on $A \otimes_K B$.

    To see that $(A \otimes_K B, \bd)$ is indeed the coproduct, suppose that $C$ is a differential $K$-algebra with morphisms $f: A \to C$ and $g: B \to C$, and let $\iota_A, \iota_B$ denote the canonical morphisms of $A, B$ into their tensor product. 
    Then, defining $f\cdot g: A \otimes_K B \to C$ by setting $(f\cdot g)(a \otimes b) = f(a)g(b)$, we obtain the unique morphism $A \otimes_K B \to C$ such that $f = (f \cdot g) \circ \iota_A$ and $g = (f \cdot g) \circ \iota B$. To verify that this is differential, observe the following:
    \[
    (f \cdot g)(\delta(a\otimes b)) = (f \cdot g)(\delta a\otimes b + a\otimes \delta b) = \delta f(a)g(b) + f(a)\delta g(b) = \delta((f \cdot g)(a\otimes b)).
    \]
    We conclude that $K\DAlg$ has all finite coproducts.
    
    Now, consider an arbitrary family $(A_i: i \in I)$ of differential $K$-algebras. For any finite subset $S \subseteq I$, let $A_S$ be the finite coproduct $\bigotimes_{i \in S} A_S$ with the derivation defined above. 
    
    Let $A$ be the directed limit of the family
    \[
    \left( \bigotimes_{i \in S} A_i : S \subseteq I \text{ finite} \right)
    \]
    where the ordering is given by inclusion of finite subsets of $I$. By \cite[Theorem IX.1.1]{MacLane1978}, $A$ is the coproduct of the family $(A_i: i \in I)$.
\end{proof}

\begin{lem}
    The forgetful functor $U: K\DAlg \to K\Alg$ is cocontinuous.
\end{lem}
\begin{proof}
    It suffices to see that $U$ preserves all small coproducts and coequalisers, which is clear from the constructions in the proof of Proposition \ref{cat_of_diff_K_algs_is_cocomplete_prop}.
\end{proof}

\section{Generalised Taylor Morphisms}
In this section, we introduce our notion of a (generalised) Taylor morphism, give classical examples of such objects, and collect basic observations and results.

\begin{defn} \label{taylor_morphism_defn}
    Let $K$ be a differential ring, and $L$ a differential $K$-algebra. 
    A \textbf{$K$-Taylor morphism for $L$} is a map $T$ which sends pairs of the form $(A, \phi)$, where $A$ is a differential ring, and $\phi: A \to K$ is a ring homomorphism, to a differential ring homomorphism $T_\phi: A \to L$, satisfying the following axioms:
    \begin{itemize}
        \item[(TM1)] For $A$ a differential ring, and $\phi: A \to K$ a differential ring homomorphism, $T_\phi = \eta_L \circ \phi$, where $\eta_L : K \to L$ is the structure map of $L$.
        \item[(TM2)] For differential rings $A, B$, and $\phi: A \to K$, $\psi: B \to K$, $\chi: A \to B$ ring homomorphisms with $\chi$ differential, if
        \[
        \begin{tikzcd}
            A \ar[r, "\phi"] \ar[d, "\chi"'] & K \\
            B \ar[ur, "\psi"']
        \end{tikzcd}
        \]
        commutes, then
        \[
        \begin{tikzcd}
            A \ar[r, "T_\phi"] \ar[d, "\chi"'] & L \\
            B \ar[ur, "T_\psi"']
        \end{tikzcd}
        \]
        also commutes.
    \end{itemize}
    We say that a differential $K$-algebra $L$ \textbf{admits a $K$-Taylor morphism} if there exists a $K$-Taylor morphism $T$ for $L$.
\end{defn}

We observe that a $K$-Taylor morphism for $L$ is a functor in the following sense:

\begin{lem} \label{k_tm_is_functor_lem}
    A $K$-Taylor morphism for $L$ is precisely a functor
    \[
    T: (U \downarrow 1_K) \to (\id_\DRing \downarrow 1_L)
    \]
    such that
    \[
    \begin{tikzcd}
        (U\downarrow 1_K) \ar[r, "T"] \ar[dr, "D"'] & (\id_\DRing \downarrow 1_L) \ar[d, "D"] \\
        & \DRing
    \end{tikzcd}
    \begin{tikzcd}
        (\id_\DRing\downarrow 1_K) \ar[r, "U"] \ar[dr, "\eta_L\circ -"'] & (U\downarrow 1_K) \ar[d, "T"]\\
        & (\id_\DRing \downarrow 1_L)
    \end{tikzcd}
    \]
    commute.
    
    The correspondence between $K$-Taylor morphisms for $L$ and such functors is given as follows: let $T$ be a $K$-Taylor morphism for $L$. Then, the functor $T: (U \downarrow 1_K) \to (\id_\DRing \downarrow 1_L)$ acts on objects by $(A, \bullet, \phi) \mapsto (A, \bullet, T_\phi)$, and on morphisms by the identity. Conversely, if $S$ is such a functor, for any differential ring $A$ and ring homomorphism $\phi: A \to K$, we set $S_\phi = \psi$, where $S(A, \bullet, \phi) = (A, \bullet, \psi)$.
\end{lem}
\begin{proof}
    This is simply an exercise in bookkeeping.
\end{proof}

\begin{rmk}
    The first commutative triangle simply requires that $T$ preserves the domain of maps, and the second says that `$T$ preserves differential ring homomorphisms', i.e. axiom (TM1). The axiom (TM2) is automatic from $T$ being a functor.
\end{rmk}

\begin{eg}[The Classical Taylor Morphism] \label{classical_tm_eg}
    Let $K$ be a $\Q$-algebra, which we regard as equipped with the $m$-tuple of trivial derivations $\bzero$. 
    We define the \textbf{classical $K$-Taylor morphism} $T$ for $(K[[\bt]], \ddbt)$ as follows: let $(A, \bdelta)$ be a differential ring, and let $\phi: A \to K$ be a ring homomorphism. Define
    \begin{align}
    T_\phi : (A, \bdelta) &\to \left(K[[\bt]], \ddbt\right) \\
    a &\mapsto \sum_{\alpha \in \N^m} \frac{\phi(\bdelta^\alpha a)}{\alpha!} \bt^\alpha.
    \end{align}
    We check that $T$ is a Taylor morphism in the generalised sense. For (TM1), let $(A, \bdelta)$ be a differential ring, and suppose that $\phi: (A, \bdelta) \to (K, \bzero)$ is a differential ring homomorphism. That is, for any $a \in A$ and $i = 1,...,m$, $\phi(\delta_i(a)) = 0$. Thus, $T_\phi(a) = \phi(a)$.

    Now suppose that $(A, \bdelta)$ and $(B, \bpartial)$ are differential rings with $\phi: A \to K$, $\psi: B \to K$ ring homomorphisms, and $\chi: A \to B$ a differential ring homomorphism such that $\phi = \psi \circ \chi$. Let $a \in A$. Then,
    \[
    T_\psi(\chi(a)) = \sum_\alpha \frac{\psi(\bpartial^\alpha(\chi(a)))}{\alpha!} \bt^\alpha = \sum_\alpha \frac{\psi(\chi(\bdelta^\alpha(a)))}{\alpha!} \bt^\alpha = T_{\psi \circ \chi}(a) = T_\phi(a).
    \]
\end{eg}

\begin{eg}[The Hurwitz Morphism \cite{Keigher1997}]\label{hurwitz_morphism_eg}
    Let $(K, \bzero)$ be a constant ring. The \textbf{ring of Hurwitz series (in $m$ variables) over $K$}, denoted $H(K)$, consists of the additive group of formal power series (in $m$ variables) $K[[\bt]]$, equipped with the multiplication defined by
    \[
    \left( \sum_\alpha a_\alpha \bt^\alpha \right)\left( \sum_\alpha b_\alpha \bt^\alpha \right) = \sum_\alpha \left(\sum_{\beta+\gamma=\alpha} \binom{\alpha}{\beta} a_\beta b_\gamma \right) \bt^\alpha.
    \]
    It is naturally equipped with the $m$-tuple of derivations $\bd_K = (\dd_{K, 1},...,\dd_{K, m})$ defined by
    \[
    \dd_{K, i} \left(\sum_\alpha a_\alpha \bt^\alpha\right) = \sum_\alpha a_{\alpha + \hat{i}} \bt^\alpha,
    \]
    where $\hat{i}$ denotes the multi-index with $1$ in the $i$th position and $0$ elsewhere. That is, $\dd_{K,i}$ acts by shifting coefficients of a series downwards once in the $i$th position. The ring $H(K)$ is equipped with a natural differential $(K, \bzero)$-algebra structure given by the standard inclusion of $K$ into its (additive group) of formal power series.

    Let $(A, \bdelta)$ be a differential ring, and $\phi: A \to K$ be a ring homomorphism. The \textbf{Hurwitz morphism} (also known as the \emph{Hurwitz mapping} or \emph{expansion}) of $\phi$ is
    \begin{align}
        H_\phi : (A, \bdelta) &\to (H(K), \bd_K) \\
        a &\mapsto \sum_\alpha \phi( \bdelta^\alpha a) \bt^\alpha.
    \end{align}
    It is easy to verify that this is indeed a $(K, \bzero)$-Taylor morphism in the generalised sense.
\end{eg}
\begin{rmk}
    If $K$ has $p>0$, then $H(K)$ is not a domain. For instance, $(t_i)^p = 0$ for any $i$. In particular, we also note that $(t_i)^p$ is distinct from the series $t_i^p$. Where $K$ is a $\Q$-algebra, $(H(K), \bd_K)$ is isomorphic to $(K[[\bt]], \ddbt)$.
\end{rmk}

\begin{eg}[The Twisted Taylor Morphism \cite{SanchezTressl2020}] \label{ttm_eg}
    For a differential $\Q$-algebra $(K, \bdelta)$, we can equip the ring of formal power series $K[[\bt]]$ with the derivations $\bdelta + \ddbt$, where $\bdelta$ acts on the coefficients of series by $\bdelta$, i.e. for any multi-index $m$,
    \[
    \bdelta^\beta \left( \sum_\alpha a_\alpha \bt^\alpha\right) = \sum_\alpha \bdelta^\beta(a_\alpha) \bt^\alpha.
    \]
    This endows $K[[\bt]]$ with derivations such that the natural inclusion of $K$ into $K[[\bt]]$ is differential.

    For a differential ring $(A, \bpartial)$ and a ring homomorphism $\phi: A \to K$, the twisted Taylor morphism of $\phi$ is given explicitly by
    \begin{align}
        T^*_\phi : (A, \bpartial) &\to \left(K[[\bt]], \bdelta + \ddbt\right) \\
        a &\mapsto \sum_{\alpha} \frac{1}{\alpha!} \sum_{\beta \leq \alpha} (-1)^{\alpha-\beta} \binom{\alpha}{\beta} \bdelta^{\alpha-\beta}(\phi(\bpartial^\beta(a))) \bt^\alpha.
    \end{align}
    The reader may refer to \cite[Section 3]{SanchezTressl2020} for details on its construction. By a straightforward direct computation, one can see that $T^*$ satisfies the axioms (TM1) and (TM2).
\end{eg}
\begin{rmk}
    If $\bdelta = \bzero$, the twisted $(K, \bdelta)$-Taylor morphism is precisely the classical Taylor morphism.
\end{rmk}

We observe that if $T$ is a $K$-Taylor morphism for $L$, then $T$ restricts naturally to the subcategory of differential $K$-algebras:

\begin{prop} \label{k_tm_preserves_k_algs_prop}
    Let $K$ be a differential ring, and $T$ a $K$-Taylor morphism for $L$. Let $A$ be a differential $K$-algebra with structure map $\eta_A: K \to A$, and let $\phi: A \to K$ be a $K$-algebra homomorphism. Then, $T_\phi: A \to L$ is a differential $K$-algebra homomorphism.
\end{prop}
\begin{proof}
    It suffices to show that $T_\phi \circ \eta_A = \eta_L$. Since $\phi$ is a $K$-algebra homomorphism, we have that
    \[
    \begin{tikzcd}
        K \ar[r, "\id"] \ar[d, "\eta_A"'] & K \\
        A \ar[ur, "\phi"']
    \end{tikzcd}
    \]
    commutes. Since $\id: K \to K$ is differential, by (TM1), $T_\id = \eta_L$. As $\eta_A$ is differential, by (TM2) we have that
    \[
    \begin{tikzcd}
        K \ar[r, "\eta_L"] \ar[d, "\eta_A"'] & L \\
        A \ar[ur, "T_\phi"']
    \end{tikzcd}
    \]
    commutes, as required.
\end{proof}

We can replace the axiom (TM1) with a number of equivalent axioms, modulo (TM2). We list some of them below:

\begin{prop}
    Let $T$ be a map of the form in Definition \ref{taylor_morphism_defn} which satisfies (TM2). The following are equivalent:
    \begin{enumerate}[label=(\roman*)]
        \item For any differential ring $A$, and differential ring homomorphism $\phi: A \to K$, $T_\phi = \eta_L \circ \phi$.
        \item For any differential ring $A$ and ring homomorphism $\phi: A \to K$, if $B$ is a differential subring of $A$ such that $\phi \rest_B$ is differential, then $T_\phi \rest_B = \eta_L \circ (\phi\rest_B)$
        \item For any differential $K$-algebra $A$ and $K$-algebra homomorphism $\phi: A \to K$, $T_\phi: A \to L$ is a differential $K$-algebra homomorphism.
        \item If $\id: K \to K$ denotes the identity, then $T_\id = \eta_L$.
    \end{enumerate}
\end{prop}
Note that condition (i) is precisely the statement of (TM1).
\begin{proof} For (i) $\Rightarrow$ (ii), let $A$ be a differential ring, $\phi: A \to K$ a ring homomorphism, and $B$ a differential subring of $A$ such that $\phi \rest_B$ is differential. Let $\iota: B \to A$ denote the inclusion, which is differential. Since $\phi \rest_B = \phi \circ \iota$, applying (TM2), we have that $T_{\phi\rest_B} = T_\phi \circ \iota = T_\phi \rest_B$. Further, as $\phi\rest_B$ is differential, applying (i) yields $T_\phi \rest_B = T_{\phi \rest_B} = \eta_L \circ (\phi \rest_B)$. The reverse implication (ii) $\Rightarrow$ (i) is trivial by taking $A = B$.

The implication (i) $\Rightarrow$ (iii) is the statement of Proposition \ref{k_tm_preserves_k_algs_prop}, and (iii) $\Rightarrow$ (iv) is trivial, as $\id: K \to K$ is a $K$-algebra homomorphism, and the only differential $K$-algebra homomorphism $K \to L$ is $\eta_L$. For (iv) $\Rightarrow$ (i), let $A$ be any differential ring with $\phi: A \to K$ a differential ring homomorphism. Trivially, $\phi = \id \circ \phi$, and applying (TM2), we have that $T_\phi = T_\id \circ \phi = \eta_L \circ \phi$ by (iv).
\end{proof}

\subsection{Connections to Differentially Large Fields}

Recall from \cite{SanchezTressl2020} that a differential field (of characteristic 0) $(K, \bdelta)$ is said to be \textbf{differentially large} if it is large as a pure field, and, for any differential field extension $(K, \bdelta) \subseteq (L, \bpartial)$, if $K$ is existentially closed in $L$ as a pure field, then $(K, \bdelta)$ is existentially closed in $(L, \bpartial)$ as a differential field. We direct the reader to the paper \cite{SanchezTressl2020} for details.

We can characterise differentially large fields in terms of Taylor morphisms in the following way:

\begin{prop}
    Let $(K, \bdelta)$ be a differential field of characteristic 0, large as a field. Then, $(K, \bdelta)$ is differentially large if and only if there exists a differential $(K, \bdelta)$-algebra $(L, \bpartial)$ such that $(L, \bpartial)$ admits a $K$-Taylor morphism $T$, and $(K, \bdelta)$ is existentially closed in $(L, \bpartial)$ as differential rings.
\end{prop}
\begin{proof}
     The forward direction is by \cite[Theorem 4.3(ii)]{SanchezTressl2020}: as $(K, \bdelta)$ is differentially large, it is existentially closed in $K((\bt))$, which admits the twisted Taylor morphism as defined in Example \ref{ttm_eg}.

     For the converse, we will use the characterisation of differentially large fields in Theorem 4.3(iv) of \cite{SanchezTressl2020}. That is, a differential field $(K, \bdelta)$ is differentially large if and only if it is large as a field, and every differentially finitely generated $K$-algebra with a $K$-point has a differential $K$-point. Let $A$ be a differentially finitely generated $K$-algebra, i.e. $A \iso K\{\bar{x}\}/I$, where $\bar{x}$ is a finite tuple of indeterminates and $I$ is some differential ideal. Let $\pi: K\{\bar{x}\} \to A$ be the quotient map. Let $\phi: A \to K$ be a $K$-point, and $\tilde\phi = \phi \circ \pi$. Since $\ker(\tilde\phi)$ is prime, it contains the radical $\sqrt{I}$ of $I$, which is also a differential ideal as $K$ is a field of characteristic 0. By the Ritt-Raudenbush basis theorem, $\sqrt{I}$ is generated as a differential radical ideal by a finite set $\Sigma$ of differential polynomials. 

     Suppose that $(L, \bpartial)$ is a differential $K$-algebra which admits a $K$-Taylor morphism $T$, and that $(K, \bdelta)$ is existentially closed in $(L, \bpartial)$. In particular, this implies that $(L, \bpartial)$ is a domain. Then, $T_{\tilde{\phi}}: K\{\bar{x}\} \to L$ is a differential $K$-algebra homomorphism, with $\ker(T_{\tilde\phi})$ containing $\sqrt{I}$. Then, $T_{\tilde{\phi}}(\bar{x}) \in L$ is a differential solution to the system $\Sigma$. By existential closure, there is a solution $\bar{a} \in K$ to $\Sigma$ as well, and hence by evaluating $\bar{x}$ at $\bar{a}$, we obtain a differential point $\psi: A \to K$.
\end{proof}

\section{Restrictions of Taylor Morphisms} \label{restrictions_section}

In this section, we will show that Taylor morphisms are completely determined by their restrictions to finitely differentially generated $K$-algebras, and in doing so answer a question of Le\'on S\'anchez in the negative.

We begin by showing that a differential field $K$ cannot admit a $K$-Taylor morphism into itself.

\begin{prop} \label{no_self_ttm}
Let $(K, \bdelta)$ be a differential field. Then $K$ does not admit a $K$-Taylor morphism.
\end{prop}
\begin{proof}
Suppose that $(K, \bdelta)$ admits a $K$-Taylor morphism $T$. 
Consider the evaluation map (at 0) $\ev: K[[\bt]] \to K$, where $K[[\bt]]$ is equipped with the derivations $\bdelta + \ddbt$. 
Observe that $T_\ev: (K[[\bt]], \bdelta + \ddbt) \to (K, \bdelta)$ is a differential $K$-algebra homomorphism. 
Since $T_\ev$ is surjective onto a field, $\ker(T_\ev)$ is a maximal ideal of $K[[\bt]]$. As $K[[\bt]]$ is a local ring with maximal ideal $(\bt)$, we find that $\ker T_\ev = (\bt)$, which is a contradiction, as $(\bt)$ is not a differential ideal.
\end{proof}

We will now answer a question of Le\'on S\'anchez negatively: it was asked whether differential largeness could be characterised by a form of Taylor morphism from $K$ to itself, perhaps after restricting to differentially finitely generated $K$-algebras. 

We know by \cite[Theorem 4.3(iv)]{SanchezTressl2020} that if $K$ is a differentially large field, then every differentially finitely generated $K$-algebra $A$ with a $K$-point $\phi: A \to K$ has a differential $K$-point as well. Essentially, we ask whether such points can be found in a `uniform' or `functorial' way similarly to in the case of Taylor morphisms. Formally:

\begin{defn} 
    Let $K$ be a differential ring, and $L$ a differential $K$-algebra. A \textbf{finite $K$-Taylor morphism for $L$} is a map $\hat{T}$ which sends pairs $(A, \phi)$, where $A$ is a differentially finitely generated $K$-algebra, and $\phi: A \to K$ is a $K$-algebra homomorphism, to a differential $K$-algebra homomorphism $\hat{T}_\phi: A \to K$, satisfying (TM1) and (TM2) (with the appropriate restriction to differentially finitely generated $K$-algebras).
\end{defn}

\begin{ques_num} \label{ftm_conj} 
Let $K$ be a differential field. Is it true that $K$ is differentially large if and only if $K$ admits a finite $K$-Taylor morphism for $K$?
\end{ques_num}


\begin{eg}
    Let $T$ be any $K$-Taylor morphism for a differential $K$-algebra $L$, and let $\hat{T}$ be its restriction to differentially finitely generated $K$-algebras. Then, $\hat{T}$ is a finite Taylor morphism.
\end{eg}

We begin by showing that if we have a finite $K$-Taylor morphism $\hat{T}$ for $L$, then we can extend its domain uniquely to all differential $K$-algebras.

\begin{defn} \index{Taylor morphism(s)!restricted}
    A \textbf{restricted $K$-Taylor morphism for $L$} is a map $\tilde{T}$ which sends pairs $(A, \phi)$, where $A$ is a differential $K$-algebra and $\phi: A \to K$ is a $K$-algebra homomorphism, to a differential $K$-algebra homomorphism $\tilde{T}_\phi: A \to K$, which satisfies the axioms (TM1) and (TM2) restricted to differential $K$-algebras and $K$-algebra homomorphisms. 
\end{defn}

\begin{prop} \label{ftm_implies_rtm}
Let $K$ be a differential field, and $L$ a differential $K$-algebra. Suppose that $L$ admits a finite $K$-Taylor morphism $\hat{T}$. Then, $L$ admits a unique restricted $K$-Taylor morphism $\tilde{T}$ whose restriction to differentially finitely generated $K$-algebras is $\hat{T}$.
\end{prop}
\begin{proof}
Let $A$ be a differential $K$-algebra, and let $\phi: A \to K$ be a $K$-algebra homomorphism. Consider the directed system $B = (B_\alpha, f_{\alpha\beta})$ of differentially finitely generated $K$-subalgebras $B_\alpha \subseteq A$ and inclusion maps $f_{\alpha\beta}: B_\alpha \to B_\beta$. For each $\alpha$, the restriction $\phi_\alpha \coloneqq \phi\rest_{B_\alpha}: B_\alpha \to K$ is a $K$-algebra homomorphism. Further, for any $\alpha, \beta$ with $B_\alpha \subseteq B_\beta$, the triangle
\[
\begin{tikzcd}
B_\alpha \ar[r, "\phi_\alpha"] \ar[d, "f_{\alpha\beta}"'] & K \\
B_\beta \ar[ur, "\phi_\beta"']
\end{tikzcd}
\]
commutes. We note also that $f_{\alpha\beta}$ is a differential $K$-algebra homomorphism. Thus, as $\hat{T}$ satisfies (TM2), we conclude that the diagram
\[
\begin{tikzcd}
B_\alpha \ar[r, "\hat{T}_{\phi_\alpha}"] \ar[d, "f_{\alpha\beta}"'] & L \\
B_\beta \ar[ur, "\hat{T}_{\phi_\beta}"']
\end{tikzcd}
\]
of differential morphisms also commutes. Set $\tilde{T}_\phi$ to be the union $\bigcup_\alpha \hat{T}_{\phi_\alpha}$ which is well defined by the commutativity of the above diagram. We now claim that $\tilde{T}$ is a restricted $K$-Taylor morphism for $L$.

We establish that (TM1) holds for $\tilde{T}$. Suppose that $\phi$ is differential. 
Let $B$ be any differentially finitely generated $K$-subalgebra of $A$. By the construction of $\tilde{T}$, we have that $\tilde{T}_\phi \rest_B = \hat{T}_{\phi\rest_B} = \eta_L \circ (\phi\rest_B)$. As this holds for arbitrary differentially finitely generated $K$-subalgebras, we have that $\tilde{T}_\phi = \eta_L \circ \phi$.

For (TM2), let $B$ be another differential $K$-algebra, $\psi: B \to K$ be a $K$-algebra homomorphism and $\chi: A \to B$ be a differential $K$-algebra homomorphism such that
\[
\begin{tikzcd}
A \ar[r, "\phi"] \ar[d, "\chi"'] & K \\
B \ar[ur, "\psi"']
\end{tikzcd}
\]
commutes. Let $a \in A$, and let $C$ be the differential $K$-subalgebra of $A$ generated by $a$. Let $D$ be the differential $K$-subalgebra of $B$ generated by $\chi(a)$. Denote the restrictions $\phi\rest_C$, $\chi \rest_C$ and $\psi\rest_D$ by $\phi_C$, $\chi_C$ and $\psi_D$, respectively. Then, the following commutes:
\[
\begin{tikzcd}
C \ar[r, "\phi_C"] \ar[d, "\chi_C"'] & K \\
D \ar[ur, "\psi_D"']
\end{tikzcd}
\]
by definition. Since $C$ and $D$ are finitely generated differential $K$-algebras, and $\chi_C$ is a differential $K$-algebra homomorphism, we also have that
\[
\begin{tikzcd}
C \ar[r, "\hat{T}_{\phi_C}"] \ar[d, "\chi_C"'] & L \\
D \ar[ur, "\hat{T}_{\psi_D}"']
\end{tikzcd}.
\]
commutes by (TM2) for $\hat{T}$. Thus, for any $a \in A$, we have by construction that $\tilde{T}_\phi(a) = T_\psi \circ \chi(a)$, and the diagram
\[
\begin{tikzcd}
A \ar[r, "\tilde{T}_\phi"] \ar[d, "\chi"'] & L \\
B \ar[ur, "\tilde{T}_\psi"']
\end{tikzcd}
\]
commutes, as required. It is clear that the restriction of $\tilde{T}$ to differentially finitely generated $K$-algebras is indeed $\hat{T}$.

For uniqueness, suppose that $\tilde{T}$ and $\tilde{S}$ are restricted $K$-Taylor morphisms for $L$ such that their restrictions $\hat{T}$ and $\hat{S}$ to differentially finitely generated $K$-algebras are equal. Let $A$ be a differential $K$-algebra and $\phi: A \to K$ a $K$-algebra homomorphism. Let $a \in A$ and let $B$ be the differential $K$-subalgebra of $A$ generated by $a$, and $\psi \coloneqq \phi \rest_B$. Since $B$ is differentially finitely generated, $T_\psi = S_\psi$ by assumption. Since the triangle
\[
\begin{tikzcd}
B \ar[r, "\psi"] \ar[d, hook] & K \\
A \ar[ur, "\phi"']
\end{tikzcd}
\]
commutes, and the inclusion map $B \subseteq A$ is differential, we have that
\[
\tilde{T}_\phi(a) = \hat{T}_\psi(a) = \hat{S}_\psi(a) = \tilde{S}_\phi(a)
\]
and $\tilde{T} = \tilde{S}$, as required.
\end{proof}

We will now extend our restricted $K$-Taylor morphism to a full $K$-Taylor morphism. Our goal is the following:

\begin{prop} \label{rtm_implies_tm}
    Let $K$ be a differential ring, and $L$ a differential $K$-algebra. Suppose $L$ admits a restricted $K$-Taylor morphism $\tilde{T}$. Then it admits a unique $K$-Taylor morphism $T$ whose restriction to differential $K$-algebras is $\tilde{T}$.
\end{prop}

\begin{nota} \index[nota]{$A \otimes B$} \index[nota]{$\phi\cdot\psi$}
    For differential rings $(A, \bdelta), (B, \bpartial)$, $A \otimes_\Z B$ denotes their tensor product over $\Z$ equipped with the derivations $\bd$ given by $\dd_i(a \otimes b) = \delta_i(a) \otimes b + a \otimes \partial_i(b)$. This is the coproduct in the category of differential rings. The canonical maps from $A, B$ into their coproduct are denoted by $\iota_A: A \to A \otimes_\Z B$ and $\iota_B : B \to A \otimes_\Z B$. 

    For $\phi: A \to C$ and $\psi: B \to C$ (differential) ring homomorphisms, by the universal property of the coproduct, there is a unique (differential) ring homomorphism denoted by $\phi\cdot\psi: A \otimes_\Z B \to C$ such that
    \[
    \begin{tikzcd}
        A \ar[d, "\iota_A"'] \ar[dr, "\phi", bend left] \\
        A \otimes_\Z B \ar[r, "\phi \cdot \psi"] & C \\
        B \ar[u, "\iota_B"] \ar[ur, "\psi", bend right]
    \end{tikzcd}
    \]
    commutes. Concretely, $(\phi\cdot\psi)(a \otimes b) = \phi(a)\psi(b)$.
\end{nota}

Throughout the following, $K$ is a differential ring, $L$ is a differential $K$-algebra, and $\tilde{T}$ is a restricted $K$-Taylor morphism for $L$. We begin by constructing suitable differential homomorphisms for each pair $(R, \phi)$, where $R$ is an arbitrary differential ring, and $\phi: R \to K$ is a ring homomorphism.

\begin{lem} \label{rtm_extends_to_all_rings}
    Let $R$ be a differential ring, and $\phi: R \to K$ a ring homomorphism. There is a unique differential ring homomorphism $T_\phi: R \to K$ such that for any differential $K$-algebra $A$ with a differential ring homomorphism $\psi: R \to A$ and a $K$-algebra homomorphism $\chi: A \to K$ with $\chi \circ \psi = \phi$, we have that $T_\phi = \tilde{T}_\chi \circ \psi$.
\end{lem}
\begin{proof}
    Let $A, \phi, \psi$ satisfy the above hypotheses. Observe that the following diagram commutes:
    \[
    \begin{tikzcd}
        & R \otimes_\Z K \ar[drr, "\phi\cdot \id_K", bend left=9] \ar[dd, "\psi\cdot\eta_A"] \\
        K \ar[ur, "\iota_K"] \ar[dr, "\eta_A"] & & R \ar[ul, "\iota_R"] \ar[dl, "\psi"'] \ar[r, "\phi"] & K \\
        & A \ar[urr, "\chi"', bend right=9]
    \end{tikzcd}
    \]
    The left square commutes by the universal property of the coproduct. Now, the lower right triangle commutes by assumption, and the upper right triangle commutes again by the universal property of coproducts. We now observe that we have the following commutative triangle of $K$-algebra homomorphisms:
    \[
    \begin{tikzcd}
        R \otimes_\Z K  \ar[d, "\psi\cdot\eta_A"'] \ar[r, "\phi \cdot \id"] & K \\
        A \ar[ur, "\chi"']
    \end{tikzcd}
    \]
    where $\psi\cdot\eta_A$ is differential. Applying $\tilde{T}$ to the above triangle, and adding maps from $R$, we obtain the diagram
    \[
    \begin{tikzcd}
        R \ar[r, "\iota_R"] \ar[dr, "\psi"'] & R \otimes_\Z K  \ar[d, "\psi\cdot\eta_A"] \ar[r, "\tilde{T}_{\phi \cdot \id_K}"] & L \\
        & A \ar[ur, "\tilde{T}_\chi"']
    \end{tikzcd}
    \]
    which commutes, as the right triangle commutes by property (TM2) of $\tilde{T}$, and the left triangle is a subdiagram of the first diagram. Define $T_\phi: R \to L$ to be $\tilde{T}_{\phi\cdot\id_K} \circ \iota_R$. Thus, we have shown that for any $A, \psi, \chi$ satisfying the hypothesis in the lemma, we have that $\tilde{T}_\chi \circ \psi = \tilde{T}_{\phi\cdot \id} \circ \iota_R = T_\phi$, as required.

    For uniqueness, let $T_\phi$ satisfy the above condition. Then, $R \otimes_\Z K$ is a differential $K$-algebra, $\iota_R: R \to R \otimes_\Z K$ is a differential ring homomorphism and $\phi \cdot \id_K: R \otimes_\Z K \to K$ is a $K$-algebra homomorphism such that $(\phi \cdot \id_K) \circ \iota_R = \phi$. Thus by the above condition, $T_\phi = \tilde{T}_{\phi\cdot\id_K} \circ \iota_R$, as required.
\end{proof}

Now, for any differential ring $A$ and ring homomorphism $\phi:A \to K$, define $T_\phi: R \to L$ as in Lemma \ref{rtm_extends_to_all_rings}, i.e.
\[
T_\phi = \tilde{T}_{\phi \cdot \id_K} \circ \iota_R.
\]
We will establish that $T$ is indeed a $K$-Taylor morphism for $L$. First, we show that (TM2) holds for $T$.

\begin{lem} \label{rtm_extends_to_all_rings_ttm2}
     Let $A, B$ be differential rings, $\phi: A \to K$, $\psi: B \to K$ ring homomorphisms, and $\chi: A \to B$ a differential ring homomorphism. If $\phi = \psi \circ \chi$, then $T_\phi = T_\phi \circ \chi$.
\end{lem}
\begin{proof}
    Let $\iota_A: A \to A\otimes_\Z K$ and $\iota_B: B \to B \otimes_\Z K$ denote the canonical maps into the tensor product. We form the following commutative diagram:
    \[
    \begin{tikzcd}
        A \ar[r, "\iota_A"] \ar[d, "\chi"'] \ar[rr, bend left=40, "\phi"]& A \otimes_\Z K \ar[d, "\chi \otimes \id_K"'] \ar[r, "\phi \cdot \id"] & K\\
        B  \ar[r, "\iota_B"] \ar[rru, bend right=70, "\psi"'] & B \otimes_\Z K \ar[ur, "\psi \cdot \id"']
    \end{tikzcd}
    \]
    Applying the restricted Taylor morphism $\tilde{T}$ to the right triangle and applying (TM2) for $\tilde{T}$, and by the definition of $T_\phi, T_\psi$, we obtain the following commutative diagram:
    \[
    \begin{tikzcd}
        A \ar[r, "\iota_A"] \ar[d, "\chi"'] \ar[rr, bend left=50, "T_\phi"]& A \otimes_\Z K \ar[d, "\chi \otimes \id_K"'] \ar[r, "\tilde{T}_{\phi \cdot \id}"] & L\\
        B  \ar[r, "\iota_B"] \ar[rru, bend right=90, "T_\psi"'] & B \otimes_\Z K \ar[ur, "\tilde{T}_{\psi \cdot \id}"']
    \end{tikzcd}
    \]
    In particular, we have that $T_\phi = T_\psi \circ \chi$.
\end{proof}

Finally, we show that (TM1) holds for $T$.

\begin{lem} \label{rtm_extends_to_all_rings_ttm1}
    Let $A$ be a differential ring, $\phi: A \to K$ a differential ring homomorphism. Let $\eta_L: K \to L$ denote the structure map. Then, $T_\phi = \eta_L \circ \phi$.
\end{lem}
\begin{proof}
    By definition, $T_\phi = \tilde{T}_{\phi\cdot\id_K} \circ \iota_R$. Since both $\phi$ and $\id_K$ are differential, $\phi\cdot\id_K: A \otimes_\Z K \to K$ is a differential ring homomorphism. In particular, it is a differential $K$-algebra homomorphism. Applying (TM1) for $\tilde{T}$, we obtain that $\tilde{T}_{\phi\cdot\id_K} = \eta_L \circ (\phi\cdot\id_K)$. Composing, we get that
    \[
    T_\phi = \eta_L \circ (\phi\cdot\id_K) \circ \iota_R = \eta_L \circ \phi
    \]
    as required.
\end{proof}

\begin{proof}
    Existence is by Lemmas \ref{rtm_extends_to_all_rings}, \ref{rtm_extends_to_all_rings_ttm2} and \ref{rtm_extends_to_all_rings_ttm1}. Uniqueness follows from the uniqueness condition in Lemma \ref{rtm_extends_to_all_rings} as any such Taylor morphism must satisfy the hypothesis in this lemma.
\end{proof}

\begin{prop} \label{ftm_implies_tm}
    Suppose $L$ admits a finite $K$-Taylor morphism $\hat{T}$. Then, it admits a unique $K$-Taylor morphism $T$ whose restriction to finitely differentially generated $K$-algebras is $\hat{T}$.
\end{prop}
\begin{proof}
    By Proposition \ref{ftm_implies_rtm} and \ref{rtm_implies_tm}.
\end{proof}

This gives a negative answer to Question \ref{ftm_conj}:

\begin{cor} \label{no_self_ftm_cor}
    Let $K$ be a differential field. Then $K$ does not admit a finite $K$-Taylor morphism.
\end{cor}
\begin{proof}
    This follows from Proposition \ref{rtm_implies_tm} and Proposition \ref{no_self_ttm}.
\end{proof}

\begin{rmk}
    We recall that if $K$ is a differentially large field, then every differentially finitely generated $K$-algebra $A$ with an algebraic $K$-point also has a differential $K$-point. Corollary \ref{no_self_ftm_cor} simply states that there is no `functorial' way to find these points.
\end{rmk}

\section{The Universal Taylor Morphism} \label{utm_section}

We apply the results of the previous section, and show that every $K$-Taylor morphism may be viewed as a cocone of a certain small diagram in the category of differential $K$-algebras. 
In doing so, we obtain the existence of a `universal' Taylor morphism, i.e. for any differential ring $K$, there exists a $K$-Taylor morphism $T^*$ such that every $K$-Taylor morphism $S$ factors uniquely through $T^*$.



    

Let $U: K\DAlg_\dfg \to K\Alg$ be the forgetful functor (into the category of all $K$-algebras). 
Consider the comma category $(U \downarrow 1_K)$, which is skeletally small. 
We let $\mathcal{D}_K$ be a skeleton of $(U \downarrow 1_K)$ which satisfies the property that for any two non-isomorphic objects $(A, \bullet, \phi)$ and $(B, \bullet, \psi)$ in $\mathcal{D}_K$, $D(A, \bullet, \phi) \neq D(B, \bullet, \psi)$, where $D: \mathcal{D}_K \to K\DAlg_\dfg$ is the restriction of the domain functor. 
This can be achieved by choosing representatives whose underlying differential $K$-algebras are pairwise distinct (but possibly isomorphic). 
That is, for each object $(A, \bullet, \phi)$ in $(U \downarrow 1_K)$, we include a representative $(A^\phi, \bullet, \phi)$, where $A^\phi$ is a copy of $A$ in $K\DAlg_\dfg$ which is labelled by $\phi$. 
In particular, if $\phi, \psi: A \to K$ are distinct morphisms, we consider $A^\phi$ and $A^\psi$ to be distinct objects in $K\DAlg_\dfg$.

We now consider the diagram $D: \mathcal{D}_K \to K\DAlg$, where $K\DAlg_\dfg$ is considered as a subcategory of $K\DAlg$ under the inclusion functor.

\begin{nota}
    We will denote a cocone of $D: \mathcal{D}_K \to K\DAlg$ by a pair $(L, \tau)$, where $L$ is the sink of the cocone, and for each object $A^\phi$ in $D(\mathcal{D}_K)$, the component at $A^\phi$ is denoted $\tau_{A^\phi}: A^\phi \to L$.
\end{nota}

\begin{prop} \label{taylor_morphism_is_cone}
    Let $K$ be a differential ring, and let $(L, \tau)$ be a cocone of $D: \mathcal{D}_K \to K\DAlg$. 
    For each pair $(A, \tilde\phi)$, where $A$ is a differentially finitely generated $K$-algebra and $\tilde\phi: A \to K$ is a $K$-algebra homomorphism, let $(A^\phi, \bullet, \phi)$ be the representative of the $(U \downarrow 1_K)$-isomorphism class of $(A, \bullet, \tilde\phi)$ in $\mathcal{D}_K$. 
    Let $(\theta, \id): (A, \bullet, \tilde\phi) \to (A^\phi, \bullet, \phi)$ be an isomorphism, and set $T_{\tilde\phi} = \tau_{A^\phi} \circ \theta: A \to L$. 
    Then, $T$ is well-defined and is a finite $K$-Taylor morphism for $L$.

    Conversely, if $T$ is a finite $K$-Taylor morphism for $L$, for each object $(A^\phi, \bullet, \phi) \in \mathcal{D}_K$, set $\tau_{A^\phi} = T_\phi: A^\phi \to L$. 
    Then, $(L, \tau)$ is a cocone for the diagram $D$.

    This gives a bijective correspondence between cocones of $\mathcal{D}_K$ and finite $K$-Taylor morphisms (and hence also $K$-Taylor morphisms).
\end{prop}
\begin{proof}
    We begin with the forward direction. Let $(L, \tau)$ be a cocone for $D$, and define $T$ as above. Let $A$ be a differentially finitely generated differential $K$-algebra, and $\tilde\phi: A \to K$ a $K$-algebra homomorphism. We first show that $T_{\tilde\phi}$ is well-defined.

    Let $(\theta, \id)$, $(\tilde\theta, \id): (A, \bullet, \tilde\phi) \to (A^\phi, \bullet, \phi)$ be isomorphisms. Then, $(\tilde\theta \circ \theta^{-1}, \id)$ is an automorphism of $(A^\phi, \bullet, \phi)$. Since $(L, \tau)$ is a cocone of $D$, we have that the triangle
    \[
    \begin{tikzcd}
        A^\phi \ar[d, "\tilde\theta \circ \theta^{-1}"'] \ar[r, "\tau_{A^\phi}"] & L \\
        A^\phi \ar[ur, "\tau_{A^\phi}"']
    \end{tikzcd}
    \]
    commutes. In particular, 
    \[
    \tau_{A^\phi} \circ \theta = (\tau_{A^\phi} \circ \tilde\theta \circ \theta^{-1}) \circ \theta = \tau_{A^\phi} \circ \tilde\theta \circ \tau_{A^\phi}
    \]
    and $T$ is well defined.


    It is clear that by applying compositions with suitable isomorphisms, it suffices to show (TM1) and (TM2) for objects in the image of $D$.
    For (TM2), let $A, B$ be differentially finitely generated $K$-algebras, and let $\phi: A \to K$, $\psi: B \to K$ be $K$-algebra homomorphisms, and $\chi: A \to B$ a differential $K$-algebra homomorphism such that $\phi = \psi \circ \chi$. Then, by construction, $(\chi, \id): (A^\phi, \bullet, \phi) \to (B^\psi, \bullet, \psi)$ is a morphism in $\mathcal{D}_K$, and the triangle
    \[
    \begin{tikzcd}
        A^\phi \ar[r, "T_\phi"] \ar[d, "\chi"'] & L \\
        B^\psi \ar[ur, "T_\psi"']
    \end{tikzcd}
    \]
    commutes by construction as $\tau$ is a cocone of $D$.

    For (TM1), observe that if $\phi: A \to K$ is differential, and that $T_\id: K \to L$ is precisely the structure map $\eta_L$, we can apply (TM2) to the triangle
    \[
    \begin{tikzcd}
        A \ar[r, "\phi"] \ar[d, "\phi"'] & L \\
        K \ar[ur, "\id"']
    \end{tikzcd}
    \]
    from which we see that $T_\phi = \eta_L \circ \phi$, as required.

    For the backwards direction, let $T$ be a (finite) Taylor morphism for $L$. Let $\tau$ be as defined above. We need to show that for any morphism $(\chi, \id): (A^\phi, \bullet, \phi) \to (B^\psi, \bullet, \psi)$ in $\mathcal{D}_K$, the triangle
    \[
    \begin{tikzcd}
        A^\phi \ar[r, "\chi"] \ar[dr, "\tau_{A^\phi}"'] & B^\psi \ar[d, "\tau_{B^\psi}"] \\
        & L
    \end{tikzcd}
    \]
    commutes. This is immediate from the axiom (TM2) for $T$, as by assumption, $\chi$ is differential with $\phi = \psi \circ \chi$, and so $T_\phi = T_\psi \circ \chi$, i.e. $\tau_{A^\phi} = \tau_{B^\psi} \circ \chi$, as required.

    The identifications between finite $K$-Taylor morphisms and cocones of $D$ are clearly mutually inverse to one another, thus we obtain the desired bijective correspondence. Further, by Proposition \ref{ftm_implies_tm}, there is a bijection between $K$-Taylor morphisms and finite $K$-Taylor morphisms. Thus, we obtain a bijective correspondence between $K$-Taylor morphisms and cocones of $D$, as required.
\end{proof}

By taking the colimit of the diagram $D$, we are able to find a `universal' $K$-Taylor morphism $T^*$ such that every other $K$-Taylor morphism factors through $T^*$.

\begin{thm} \label{universal_TM_construction_thm} 
    For any differential ring $K$, there is a differential $K$-algebra $K^*$ and $K$-Taylor morphism $T^K$ for $K^*$ such that for any differential $K$-algebra $L$ admitting a $K$-Taylor morphism $T$, there is a unique differential $K$-algebra homomorphism $\theta: K^* \to L$ such that for any differential $K$-algebra $A$ and $K$-algebra homomorphism $\phi: A \to K$, we have that $T_\phi = \theta \circ T^K_\phi$.
\end{thm}
\begin{proof}
    As the category of differential $K$-algebras is cocomplete, the diagram $D$ has a initial cone $(K^*, \tau^*)$, and let $\hat{T}^K$ be the finite $K$-Taylor morphism associated to this cocone given by Proposition \ref{taylor_morphism_is_cone}. By Proposition \ref{ftm_implies_tm}, there is a unique $K$-Taylor morphism $T^K$ whose restriction to differentially finitely generated $K$-algebras is $\hat{T}^K$.

    Let $L$ be any differential $K$-algebra admitting a $K$-Taylor morphism $T$. Let $(L, \tau)$ be the cocone of $D$ associated to $T$. As $(K^*, \tau^*)$ is the initial cocone, $(L, \tau)$ factors through $(K^*, \tau^*)$. In particular, there is a unique differential $K$-algebra morphism $\theta: K^* \to L$ such that for any differentially finitely generated $K$-algebra $A$ and $K$-algebra homomorphism $\phi: A \to K$, we have that $T_\phi = \theta \circ T^K_\phi$.

    Now, observe that for any differential $K$-algebra $A$ (not necessarily differentially finitely generated) and $K$-algebra homomorphism $\phi$, $T_\phi$ is completely determined by all of its restrictions to differentially finitely generated $K$-subalgebras of $A$. Thus we still have $T_\phi = \theta \circ T^K_\phi$ in this case.

    Finally, recall from Lemma \ref{rtm_extends_to_all_rings} that for any differential ring $R$ and ring homomorphism $\phi: R \to K$, we have that $T_\phi = T_\chi \circ \psi$, where $\psi: R \to A$ is a differential ring homomorphism, $A$ is a differential $K$-algebra, and $\chi: A \to K$ satisfies $\chi \circ \psi = \phi$. Since $A$ is a differential $K$-algebra, we can apply the above to obtain that $T_\chi = \theta \circ T^K_\chi$, thus $T_\phi = \theta \circ T^K_\chi \circ \psi$. By (TM2) for $T^K$, $T^K_\chi \circ \psi = T^K_\phi$, therefore we recover that $T_\phi = \theta \circ T^K_\phi$, as required.
\end{proof}

\begin{defn} 
    For a differential ring $K$, the \textbf{universal $K$-Taylor morphism} is the $K$-Taylor morphism $T^K$ for $K^*$ as constructed in Theorem \ref{universal_TM_construction_thm}.
\end{defn}

\begin{defn} \label{def_morphs_of_tms}
    Let $K$ be a differential ring, and let $T$ and $S$ be $K$-Taylor morphisms for $L, F$, respectively. A \textbf{morphism of $K$-Taylor morphisms} $\theta: T \to S$ is a $K$-algebra homomorphism $\theta: L \to F$ such that for any differential ring $A$ and ring homomorphism $\phi: A \to K$, $S_\phi = \theta \circ T_\phi$. We say that $\theta: T \to S$ is an \textbf{isomorphism} of $K$-Taylor morphisms if the underlying morphism $\theta: L \to F$ of differential $K$-algebras is an isomorphism. The \textbf{category of $K$-Taylor morphisms} is the category whose objects are $K$-Taylor morphisms and morphisms are as above.
\end{defn}

\begin{rmk}
By definition, the universal $K$-Taylor morphism is the initial object in the category of $K$-Taylor morphisms.
\end{rmk}

\begin{rmk}
    The category of $K$-Taylor morphisms is isomorphic to the category of differential $K^*$-algebras, by identifying each $K$-Taylor morphism with its underlying $K$-algebra.
\end{rmk}

\section{Evaluation Maps} \label{eval_maps_section}

In this section, we discuss the existence of inverses for certain Taylor morphisms. These take the form of a generalised `evaluation map', which capture certain properties of the `evaluation at 0' map $\ev_0: K[[\bt]] \to K$. We will show that the existence of such a map characterises the universal Taylor morphism.

\begin{defn}\label{eval_maps_defn}
	Let $K$ be a differential ring, $L$ a differential $K$-algebra, and $T$ a $K$-Taylor morphism for $L$. An \textbf{evaluation map for $T$} is a $K$-algebra homomorphism $\ev: L \to K$ satisfying the following:
	\begin{itemize}[align=left]
	\item[(EV1)] For any differential ring $A$ and ring homomorphism $\phi: A \to K$, 
	\[
	\ev \circ T_\phi = \phi.
	\]
	\item[(EV2)] $T_\ev: L \to L$ is the identity on $L$.
	\end{itemize}
	If an evaluation map for $T$ exists, we say that $T$ \textbf{admits an evaluation map}.
\end{defn}

\begin{egs} \label{eval_map_egs}
    \begin{enumerate}
        \item \index[nota]{$\ev_0$}Let $K$ be a constant $\Q$-algebra, and let $T$ be the classical Taylor morphism for $(K[[\bt]], \ddbt)$. Then, the `evaluation at 0' map,
        \[
        \ev_0: \sum_\alpha a_\alpha \bt^\alpha \mapsto a_{\bar0}
        \]
        where $\bar{0} = (0,...,0)$, is an evaluation map for $T$. This is easy to verify: for (EV1), let $(A, \bdelta)$ be a differential ring, $\phi: A \to K$ be a ring homomorphism and let $a \in A$. Then,
        \[
        \ev_0(T_\phi(a)) = \ev_0\left( \sum_\alpha \frac{\phi(\delta^\alpha(a))}{\alpha!} \bt^\alpha\right) = \phi(a).
        \]
        For (EV2), let $a= \sum_\alpha a_\alpha \bt^\alpha \in K[[\bt]]$, and compute $T_{\ev_0}(a)$:
        \begin{align}
            T_{\ev_0}(a) &= \sum_\alpha \frac{1}{\alpha!}\ev_0\left(\left(\frac{\dd}{\dd \bt}\right)^\alpha a\right) \bt^\alpha \\
            &= \sum_\alpha \frac{1}{\alpha!} \ev_0 \left(\sum_\beta \frac{(\beta+\alpha)!}{\beta!} a_{\beta+\alpha} \bt^\beta \right) \bt^\alpha \\
            &= \sum_\alpha \frac{\alpha! a_\alpha}{\alpha!} \bt^\alpha = \sum_\alpha a_\alpha \bt^\alpha = a
        \end{align}
        so $T_{\ev_0} = \id$, as required.
        
        \item Let $(K, \bdelta)$ be a $\Q$-algebra, and let $T^*$ be the twisted Taylor morphism for $(K[[\bt]], \bdelta + \ddbt)$. Then $\ev_0: K[[\bt]] \to K$ is also an evaluation map for $T^*$. We verify this as follows: to see that (EV1) holds, let $A$ be a differential ring, $\phi: A \to K$ be a ring homomorphism and $a \in A$. Observe by applying the explicit formula in Example \ref{ttm_eg}, that $\ev_0\circ T^*_\phi(a)$, i.e. the coefficient of $\bt^{\bar{0}}$, is precisely $\phi(a)$. For (EV2), let $\sum_\alpha a_\alpha \bt^\alpha \in K[[\bt]]$, and write $T^*_{\ev_0}(\sum_\alpha a_i \bt^\alpha) = \sum_\alpha b_\alpha \bt^\alpha$. We compute $b_\alpha$:
        \[
            b_\alpha = \frac{1}{\alpha!} \sum_{\beta \leq \alpha} (-1)^{\alpha-\beta} \binom{\alpha}{\beta} \left( \ev_0 \left( \left(\bdelta + \ddbt \right)^\beta \left( \sum_\xi a_\xi \bt^\xi \right) \right) \right).
        \]
        Observe that
        \begin{align}
            \left(\bdelta + \ddbt\right)^\beta\left(\sum_\xi a_\xi \bt^\xi\right) &= \sum_{\gamma \leq \beta} \binom{\beta}{\gamma} \bdelta^{\beta-\gamma} \left(\ddbt\right)^\gamma \left( \sum_\xi a_\xi \bt^\xi \right) \\
            &= \sum_{\gamma \leq \beta} \sum_\xi \frac{(\xi+\gamma)!}{\xi!} \binom{\beta}{\gamma} \bdelta^{\beta-\gamma}(a_{\xi+\gamma}) t^\xi.
        \end{align}
        As $\ev_0((\bdelta + \ddbt)^\beta(\sum_\xi a_\xi \bt^\xi))$ is the coefficient of $\bt^{\bar0}$ in the above sum, we see that
        \[
        \ev_0\left(\left(\bdelta + \ddbt\right)^\beta\left(\sum_\xi a_\xi t^\xi\right)\right) = \sum_{\gamma \leq \beta} \gamma! \binom{\beta}{\gamma} \bdelta^{\beta-\gamma}(a_\gamma).
        \]
        Substituting this into the equation for $b_\alpha$, we obtain that
        { \allowdisplaybreaks
        \begin{align}
            b_\alpha &= \frac{1}{\alpha!} \sum_{\beta \leq \alpha} (-1)^{\alpha-\beta} \binom{\alpha}{\beta} \left( \sum_{\gamma \leq \beta} \gamma! \binom{\beta}{\gamma} \bdelta^{\beta-\gamma}(a_\gamma)\right) \\
            &= \sum_{\beta \leq \alpha} \sum_{\gamma \leq \beta} \frac{(-1)^{\alpha-\beta}\gamma!}{\alpha!} \binom{\alpha}{\beta} \binom{\beta}{\gamma} \bdelta^{\alpha-\gamma}(a_\gamma).
        \end{align}
        }
        Fix some $\gamma \leq \alpha$. Then, the coefficient $c_{\alpha, \gamma}$ of $\bdelta^{\alpha-\gamma}(a_\gamma)$ in the above sum (ranging over $\beta$) is
        \begin{align}
            \sum_{\gamma \leq \beta \leq \alpha} \frac{(-1)^{\alpha-\beta} \gamma!}{\alpha!} \binom{\alpha}{\beta} \binom{\beta}{\gamma} = \sum_{\gamma \leq \beta \leq \alpha} \frac{(-1)^{\alpha-\beta}}{(\alpha-\beta)!(\beta-\gamma)!} \\
            = \frac{(-1)^{\alpha+\gamma}}{(\alpha-\gamma)!} \sum_{\gamma \leq \beta \leq \alpha} \frac{(-1)^{\beta-\gamma} (\alpha-\gamma)!}{(\beta-\gamma)!((\alpha-\gamma)-(\beta-\gamma))!}.
        \end{align}
        Let $\epsilon = \beta-\gamma$ and reindex the sum, obtaining:
        \[
        c_{\alpha, \gamma} = \frac{(-1)^{\alpha+\gamma}}{(\alpha-\gamma)!} \sum_{\epsilon\leq \alpha-\gamma} (-1)^\epsilon \binom{\alpha-\gamma}{\epsilon}.
        \]
        The alternating sum of binomial coefficients is 0 unless $\alpha=\gamma$, in which case it is 1. Observe then, that $c_{\alpha, \gamma} = 0$ if $\alpha \neq \gamma$ and $c_{\alpha, \alpha} = 1$. Finally, we see the following:
        \[
        b_\alpha = \sum_{\gamma \leq \alpha} c_{\alpha, \gamma} \bdelta^{\alpha-\gamma}(a_\gamma) = c_{\alpha, \alpha} a_\alpha = a_\alpha
        \]
        and $T^*_{\ev_0}(\sum_\alpha a_\alpha \bt^\alpha) = \sum_\alpha a_\alpha \bt^\alpha$, and $T^*_{\ev_0} = \id$, as required.

        \item Let $K$ be an arbitrary constant ring, and let $H$ be the Hurwitz morphism for $(H(K), \bd_K)$ as constructed in Example \ref{hurwitz_morphism_eg}. Then, the map $\ev: H(K) \to K$ given by $\sum_\alpha a_\alpha \bt^\alpha \mapsto a_{\bar0}$ is a evaluation map for $H$. We verify this by direct computation: let $(A, \bdelta)$ be a differential ring, and let $\phi: A \to K$ be a ring homomorphism and let $a \in A$. Then,
        \[
        \ev(H_\phi(a)) = \ev\left(\sum_\alpha \phi(\bdelta^\alpha a) \bt^\alpha \right) = \phi(a).
        \]
        Thus (EV1) holds. Note that (EV1) is also the conclusion of \cite[Proposition 2.1]{Keigher1997}. Now, for (EV2), let $\sum_\alpha a_\alpha \bt^\alpha \in H(K)$. Then, writing $H_{\ev}(\sum_\alpha a_\alpha \bt^\alpha)$ as $\sum_\alpha b_\alpha \bt^\alpha$, we have that for any fixed $\alpha$,
        \[
        b_\alpha = \ev\left(\bd_K^\alpha \left(\sum_\beta a_\beta \bt^\beta \right)\right) = \ev\left(\sum_\beta a_{\beta+\alpha} \bt^\beta \right) = a_\alpha
        \]
        So $H_\ev = \id_{H(K)}$, as required.
    \end{enumerate}
\end{egs}

For the remainder of this section, we let $K$ be a differential ring, $L$ be a differential $K$-algebra and $T$ a $K$-Taylor morphism for $L$.

\begin{lem} \label{ev_commutes_tm}
	 Let $\ev: L \to K$ be a $K$-algebra homomorphism. Then, $T_\ev = \id_L$ if and only if for any differential ring $B$ and differential ring homomorphism $\psi: B \to L$, $T_{\ev\,\circ\,\psi} = \psi$. 
\end{lem}
\begin{proof}
	For the forward direction, let $T_\ev = \id_L$. Let $B$ be a differential ring, and $\psi: B \to L$ be a differential ring homomorphism. Then, the following commutes:
	\[
	\begin{tikzcd}
	B \ar[r, "\ev \, \circ \, \psi"] \ar[d, "\psi"'] & K\\
	L \ar[ur, "\ev"']
	\end{tikzcd}
	\]
	As $\psi$ was assumed to be differential, we apply (TM2) to obtain that the following also commutes:
	\[
	\begin{tikzcd}
	B \ar[r, "T_{\ev\, \circ\, \psi}"] \ar[d, "\psi"'] & K\\
	L \ar[ur, "T_\ev"']
	\end{tikzcd}
	\]
	By assumption, $T_\ev = \id_L$. Thus, $T_{\ev\, \circ\, \psi} = \id_L\circ \psi = \psi$. 
	
	Conversely, assume that for any differential ring $B$ and differential ring homomorphism $\psi: B \to L$, $T_{\ev\,\circ\,\psi} = \psi$. Then, consider the following commutative diagram:
	\[
	\begin{tikzcd}
	L \ar[r, "\ev\,\circ\,\id_L"] \ar[d, "\id_L"'] & K\\
	L \ar[ur, "\ev"']
	\end{tikzcd}
	\]
	As $\id_L$ is differential, we apply (TM2) to obtain that
	\[
	\begin{tikzcd}
	L \ar[r, "T_{\ev\,\circ\,\id_L}"] \ar[d, "\id_L"'] & L\\
	L \ar[ur, "T_\ev"']
	\end{tikzcd}
	\]
	also commutes. By assumption, $T_{\ev\,\circ\,\id_L} = \id_L$. Thus, $T_\ev \circ \id_L = T_\ev = T_{\ev\,\circ\,\id_L} = \id_L$, as required.
\end{proof}

\begin{lem}
	If $T$ admits an evaluation map, then it is unique.
\end{lem}
\begin{proof}
	Suppose both $\ev$ and $\ev'$ are evaluation maps for $T$. Then, by (EV1) on $\ev$, we have that $\ev\circ T_{\ev'} = \ev'$. By (EV2) on $\ev'$, we have that $T_{\ev'} = \id_L$. Thus, 
	\[
	\ev' = \ev\circ T_{\ev'} = \ev\circ\id_L = \ev
	\]
	as required.
\end{proof}

\begin{prop}
    Suppose $T$ admits an evaluation map $\ev$. Define the functor
    \[
    \ev \circ - : (\id_\DRing \downarrow 1_L) \to (U\downarrow 1_K)
    \]
    by setting
    \[
    (\ev \circ -)(A, \bullet, \phi) = (A, \bullet, \ev \circ \phi)
    \]
    on objects, and $(\ev \circ -)$ to be the identity on morphisms. Then, $(\ev \circ -)$ and $T$ are inverse functors, and $(U \downarrow 1_K)$ and $(\id_\DRing \downarrow 1_L)$ are isomorphic as categories.
\end{prop}
\begin{proof}
    We begin by verifying that $(\ev \circ -)$ is a functor. We check that if $(\alpha, \id): (A, \bullet, \phi) \to (B, \bullet, \psi)$ is a morphism in $(\id_\DRing \downarrow 1_L)$, then it is also a morphism $(A, \bullet, \ev \circ \phi) \to (B, \bullet, \ev\circ \psi)$ in $(U \downarrow 1_K)$.
    Consider the following diagram:
    \[
    \begin{tikzcd}
        A \ar[r, "\phi"] \ar[d, "\alpha"] & L \ar[d, "\id_L"] \ar[r, "\ev"] & K \ar[d, "\id_K"] \\
        B \ar[r, "\psi"] & L \ar[r, "\ev"] & K 
    \end{tikzcd}
    \]
    The left square commutes by definition of morphisms in $(\id_\DRing \downarrow 1_L)$. The right square is clearly commutative also. Thus, the large square commutes (considered as a diagram in $\Ring$), and $(\alpha, \id)$ is a morphism $(A, \bullet, \ev \circ \phi) \to (B, \bullet, \ev\circ \psi)$ in $(U \downarrow 1_K)$. Clearly $(\ev \circ -)$ preserves composition.

    Now, let $(A, \bullet, \phi)$ be an object in $(\id_\DRing \downarrow 1_L)$. Then, 
    \[
    T(\ev \circ -)(A, \bullet, \phi) = (A, \bullet, T_{\ev\circ\phi}) = (A, \bullet, \phi)
    \]
    by Lemma \ref{ev_commutes_tm}. Conversely, let $(B, \bullet, \psi)$ be an object in $(U \downarrow 1_K)$. Then,
    \[
    (\ev \circ -)T(B, \bullet, \psi) = (B, \bullet, \ev\circ T_\psi) = (B, \bullet, \psi)
    \]
    by (EV1). As both $T$ and $(\ev \circ -)$ are the identity on morphisms, we obtain the desired result.
\end{proof}

\begin{prop} \label{utm_admits_eval_map_prop}
    Let $K$ be a differential ring, and let $T^K$ be the universal $K$-Taylor morphism for $K^*$. Then $T^K$ admits an evaluation map $\ev_K: K^* \to K$.
\end{prop}
\begin{proof}
    Let $F: \mathcal{D}_K \to K\Alg/K = (\id_{K\Alg} \downarrow 1_K)$ be the `forgetful functor', i.e. $U$ maps an object $(A, \bullet, \phi)$ in $\mathcal{D}_K$ to $(U(A^\phi), \phi)$. 
    By Theorem \ref{certain_comma_category_is_cocomplete_thm}, since the categories $K\Alg$ and $\mathbf{1}$ are cocomplete, and the identity is clearly cocontinuous, we have that $K\Alg/K$ is also cocomplete. Let the colimit of $F$ be $(L, \ev_K)$, and $\tau_{A^\phi}: (U(A^\phi), \phi) \to (L, \ev_K)$ be the component of the colimiting cocone at $(U(A^\phi), \phi)$.

    We claim that $L$ is precisely $K^*$, and $\tau_{A^\phi} = T^K_\phi$ for each $A^\phi$.
    By Theorem \ref{certain_comma_category_is_cocomplete_thm} the domain functor $K\Alg/K \to K\Alg$ is cocontinuous, and colimit of $DU$ is $L$, with the component at $A^\phi$ precisely $\tau_{A^\phi}$.

    By construction, the colimit of the diagram $D: \mathcal{D}_K \to K\DAlg$ is $K^*$, with components $T^K_\phi: A_\phi \to K^*$. 
    Since the forgetful functor $U: K\DAlg \to K\Alg$ is cocontinuous, we have that the colimit of $UD: \mathcal{D}_K \to K\Alg$ is also $K^*$, with component $T^K_\phi$ at $A^\phi$. 
    Observe that $UD = DF$, and we conclude that the colimiting cocones $(L, \tau)$ and $(K^*, T^K)$ are isomorphic.

    We claim that $\ev_K$ is an evaluation map for $T^K$. We begin by showing (EV1) in the special case of differential $K$-algebras and $K$-algebra homomorphisms. For any $(A^\phi, \bullet, \phi)$ in $\mathcal{D}_K$, we observe that the diagram
    \[
    \begin{tikzcd}
        A \ar[r, "T^K_\phi"] \ar[d, "\phi"'] & K^* \ar[dl, "\ev_K"] \\
        K
    \end{tikzcd}
    \]
    commutes, by definition of morphisms in the slice category $K\Alg/K$. 
    After factoring through an appropriate isomorphism if necessary, we see that for any differentially finitely generated differential $K$-algebra $A$, and $K$-algebra homomorphism $\phi: A \to K$, we have that $\ev_K \circ T^K_\phi = \phi$. 
    By taking appropriate restrictions to differentially finitely generated subalgebras, this also holds for any differential $K$-algebra.

    Now, we extend the result to all differential rings. 
    From the proof of Lemma \ref{rtm_extends_to_all_rings}, we may recover the Taylor morphism $T^K$ from its restriction to differential $K$-algebras. Explicitly, for $R$ a differential ring, and $\phi: R \to K$ be a differential ring homomorphism,
    we have that
    $T^K_\phi = T^K_{\phi \cdot \id_K} \circ \iota_R$, where $\iota_R: R \to R \otimes_\Z K$ is the natural inclusion, and $\phi \cdot \id_K: R \otimes_\Z K \to K$ is the map given by the product of $\phi$ and $\id_K$.
    Then,
    \[
    \ev_K \circ T^K_\phi = \ev_K \circ T^K_{\phi \circ \id_K} \circ \iota_R = (\phi \cdot \id_K) \circ \iota_R = \phi
    \]
    as required.

    For (EV2), let $A$ be a differentially finitely generated $K$-algebra, and let $\phi: A \to K$ be a $K$-point. By (EV1), the triangle
    \[
    \begin{tikzcd}
        A \ar[r, "\phi"] \ar[d, "T^K_\phi"'] & K \\
        K^* \ar[ur, "\ev_K"']
    \end{tikzcd}
    \]
    commutes. Then, applying $T^K$, we have that
    \[
    \begin{tikzcd}
        A \ar[r, "T^K_\phi"] \ar[d, "T^K_\phi"'] & K \\
        K^* \ar[ur, "T^K_{\ev_K}"']
    \end{tikzcd}
    \]
    also commutes. This implies that $T^K_{\ev_K}$ is an automorphism of the colimiting cocone $(K^*, T^K)$ of the diagram $D: \mathcal{D}_K \to K\DAlg$. By the universal property of colimits, we obtain that $T^K_{\ev_K} = \id_{K^*}$, as required.
\end{proof}

\begin{lem} \label{tm_evalmap_is_morphism_of_tms_lem}
    Let $T, S$ be $K$-Taylor morphisms for $L, F$, respectively. Suppose $T$ admits an evaluation map $\ev: L \to K$. Then, $S_\ev: L \to F$ is a morphism of $K$-Taylor morphisms $T \to S$. Further, if $\theta: T \to S$ is any other morphism, then $\theta = S_\ev$.
\end{lem}
\begin{proof}
    Let $A$ be a differential ring, and $\phi: A \to K$ be a differential ring homomorphism. By (EV1), $\ev \circ T_\phi = \phi$. That is, the following triangle commutes:
    \[
    \begin{tikzcd}
        A \ar[r, "\phi"] \ar[d, "T_\phi"'] & K \\
        L \ar[ur, "\ev"']
    \end{tikzcd}
    \]
    Applying $S$ and (TM2), we obtain that
    \[
    \begin{tikzcd}
        A \ar[r, "S_\phi"] \ar[d, "T_\phi"'] & F \\
        L \ar[ur, "S_\ev"']
    \end{tikzcd}
    \]
    also commutes, and thus $S_\ev \circ T_\phi = S_\phi$, as required.

    For uniqueness, let $\theta: T \to S$ be any other morphism of $K$-Taylor morphisms. Then, by definition, for any differential ring $A$ and ring homomorphism $\phi: A \to K$, we have that $S_\phi = \theta \circ T_\phi$. In particular, applying this to $\ev: L \to K$, we have that
    \[
    S_\ev = \theta \circ T_\ev = \theta \circ \id_L = \theta
    \]
    by (EV2), as required.
\end{proof}

In other words, we have that the existence of an evaluation map characterises the universal $K$-Taylor morphism:

\begin{thm} \label{universal_tm_char_by_ev_thm}
    Let $K$ be a differential ring, and $T$ a $K$-Taylor morphism. Then, there exists a unique isomorphism $T \to T^K$ if and only if $T$ admits an evaluation map.
\end{thm}
\begin{proof}
    The forward direction is Lemma \ref{tm_evalmap_is_morphism_of_tms_lem}, and the backwards direction is Proposition \ref{utm_admits_eval_map_prop}.
\end{proof}

From the examples in \ref{eval_map_egs}, we obtain the following concrete descriptions of the universal Taylor morphism in certain cases:

\begin{cor} \label{utm_for_Q_algs_and_constants_cor}
    The universal Taylor morphism for differential $\Q$-algebras and constant rings are precisely the twisted Taylor morphism and Hurwitz morphism, respectively.
\end{cor}

\subsection{Differentially Algebraic Power Series}

For the remainder of this section, we let $(K, \delta)$ be a differential field of characteristic zero in one derivation, which is large as a field. Let $K[[t]]_\dalg$ denote the differential subring of $(K[[t]],  \delta + \ddt)$ which consists of elements differentially algebraic over $(K, \delta)$. We will also write $\delta + \ddt$ for the restriction of the derivation to $K[[t]]_\dalg$.

In \cite[Proposition 3.3]{SanchezTressl2023}, Le\'on S\'anchez and Tressl prove the following using an approximation result: 
suppose that $A$ is a differentially finitely generated $K$-algebra, and suppose that there is a $K$-algebra homomorphism $A \to L$, where $L/K$ is a field extension such that $K$ is existentially closed in $L$. 
Then, there is a differential $K$-algebra homomorphism $A \to (K[[t]]_\dalg, \delta+\ddt)$.

In particular, if there is a $K$-algebra homomorphism $A \to K$, then there is a differential $K$-algebra homomorphism $A \to K[[t]]_\dalg$. It is then natural to ask whether $(K[[t]]_\dalg, \delta + \ddt)$ admits a $K$-Taylor morphism.

\begin{prop} \label{diff_alg_power_series_no_tm_prop}
    Let $(K, \delta)$ be a differential field of characteristic zero in one derivation. Then, $(K[[t]]_\dalg, \delta + \ddt)$ admits a $K$-Taylor morphism if and only if $K[[t]]_\dalg = K[[t]]$.
\end{prop}
\begin{proof}
    The reverse direction is clear. For the forward direction, suppose that $(K[[t]]_\dalg, \delta + \ddt)$ admits a $K$-Taylor morphism $T$. 
    Then, there is a unique morphism of $K$-Taylor morphisms $\theta: T^* \to T$, where $T^*$ is the twisted Taylor morphism for $(K[[t]], \delta + \ddt)$. In particular, $\theta: K[[t]] \to K[[t]]_\dalg$ is a differential $K$-algebra homomorphism.

    Observe that the only ideals of $K[[t]]$ are of the form $t^nK[[t]]$ for some $n \in \N$, thus $(K[[t]], \delta + \ddt)$ has no proper differential ideals, so $\theta$ is injective. Suppose that there is $a \in K[[t]]$ differentially transcendental over $K$. As $\theta(a)$ is differentially algebraic over $K$, there is a nonzero differential polynomial $f$ over $K$ such that $f(\theta(a)) = 0$. Since $\theta$ is a $K$-algebra homomorphism, and $\theta: K[[t]] \to \theta(K[[t]]) \subseteq K[[t]]_\dalg$ is an isomorphism, we also have that $f(a) = 0$, a contradiction.
\end{proof}

\begin{rmk}
    It is not in general known whether there exist elements in $(K[[t]], \delta + \ddt)$ differentially transcendental over $(K, \delta)$.
\end{rmk}

\section{An Adjunction}

In \cite{Keigher1975}, Keigher shows that the map $H: \Ring \to \DRing$ which sends a ring $A$ to its ring of Hurwitz series $H(A)$ (in a single variable $t$), and sends morphisms $\phi: A \to B$ to
\begin{align}
H(\phi) : H(A) &\to H(B) \\
\sum a_i t^i &\mapsto \sum \phi(a_i) t^i
\end{align}
is a functor, and is the right adjoint to the forgetful functor $U: \DRing \to \Ring$. 

A related result appears in \cite{Umemura1996} for the case of differential $\Q$-algebras in one derivation. Umemura shows that for a differential $\Q$-algebra $A$ and a $\Q$-algebra $B$, there is a bijection
\[
\Phi: \Ring(UA, B) \to \DRing(A, B[[t]])
\]
which sends a morphism $\phi: UA \to B$ to $T_\phi: A \to (B[[t]], \ddt)$, where $T$ is the classical Taylor morphism.

In this section, we will show that the universal Taylor morphism constructed in Section \ref{utm_section} has a similar adjunction property.

\begin{defn}
    The category $\DRing_\Ring$ is the category whose objects are differential rings, and morphisms $(A, \bdelta) \to (B, \bpartial)$ are ring homomorphisms $A \to B$. Write $I: \DRing \to \DRing_\Ring$ for the natural inclusion.
\end{defn}

\begin{lem} \label{UTM_is_functor_lemma}
    Let $K, L$ be differential rings, and let $\theta: K \to L$ be a ring homomorphism. There exists a unique differential ring homomorphism $\theta^*: K^* \to L^*$ such that for any differential ring $A$ and ring homomorphism $\phi: A \to K$, we have that $T^L_{\theta \circ \phi} = \theta^* \circ T^K_\phi$. 
    
    Further, the map $F: \DRing_\Ring \to \DRing$ which sends objects by $K \mapsto K^*$ and morphisms by $(\theta: K \to L) \mapsto (\theta^*: K^* \to L^*)$ is a functor.
\end{lem}
\begin{proof}
    We claim that $\theta^* \coloneqq T^L_{\theta\, \circ\, \ev_K}$ is such a map. Let $A$ be a differential ring, and $\phi: A \to K$ be a ring homomorphism. Consider the following triangle:
    \[
    \begin{tikzcd}
        A \ar[r, "\theta\, \circ\, \phi"] \ar[d, "T^K_\phi"'] & L \\
        K^* \ar[ur, "\theta\, \circ\, \ev_K"']
    \end{tikzcd}
    \]
    which commutes, as $\theta \circ \ev_K \circ T^K_\phi = \theta \circ \phi$, as $\ev_K$ is the evaluation map for $T^K$. Further, $T^K_\phi$ is differential, so applying $T^L$, we obtain that
    \[
    \begin{tikzcd}
        A \ar[r, "T^L_{\theta\, \circ\, \phi}"] \ar[d, "T^K_\phi"'] & L^* \\
        K^* \ar[ur, "T^L_{\theta\, \circ\, \ev_K}"']
    \end{tikzcd}
    \]
    also commutes, i.e. $\theta^* \circ T^K_\phi = T^L_{\theta\, \circ \,\phi}$ as required.

    For uniqueness, suppose that $\tilde\theta: K^* \to L^*$ is another such map. Then, in particular, 
    \[
    \tilde\theta = \tilde\theta \circ \id_{K^*} = \tilde\theta \circ T^K_{\ev_K} =  T^L_{\theta\, \circ\, \ev_K} = \theta^*
    \]
    as required. 
    
    For functoriality, let $K, L, F$ be differential rings, and let $\theta: K \to L$ and $\tau: L \to F$ be ring homomorphisms. Let $A$ be a differential ring and $\phi: A \to K$ be a ring homomorphism. Observe:
    \[
    \tau^* \circ \theta^* \circ T^K_\phi = \tau^* \circ T^L_{\theta\, \circ\, \phi} = T^F_{\tau\, \circ\, \theta\, \circ\, \phi}.
    \]
    Thus by uniqueness, $\tau^* \circ \theta^* = (\tau \circ \theta)^*$, as required.
\end{proof}

\begin{rmk}
    Suppose $K, L$ are differential rings, and $\phi: K \to L$ is an isomorphism of their underlying rings. Then, $\phi^* = T^L_{\phi \,\circ\, \ev_K}: K^* \to L^*$ is an isomorphism of differential rings (cf. \cite[Corollary 3.3]{SanchezTressl2020}).
\end{rmk}

\begin{thm} \label{utm_functor_is_right_adjoint_thm}
    The functor $F: \DRing_\Ring \to \DRing$ is the right adjoint to $I: \DRing \to \DRing_\Ring$. For objects $A \in \DRing$ and $K \in \DRing_\Ring$, the natural bijection of hom-sets is given by
    \begin{align}
        T^K: \DRing_\Ring(IA, K) &\to \DRing(A, K^*) \\
        (\phi: IA \to K) &\mapsto (T^K_\phi: A \to K^*)
    \end{align}
    with inverse
    \begin{align}
        \ev_K \circ - : \DRing(A, K^*) &\to \DRing_\Ring(IA, K)\\
        (\psi: A \to K^*) &\mapsto (\ev_K \circ \psi: IA \to K).
    \end{align}
\end{thm}
\begin{proof}
    The maps above are mutually inverse by Lemma \ref{ev_commutes_tm}. We verify that the bijection is natural:

    Fix morphisms $\theta: K \to L$ in $\DRing_\Ring$ and $\chi: B \to A$ in $\DRing$. Then, for any ring homomorphism $\phi: IA \to K$,
    \[
    \theta^* \circ T^K_\phi = T^L_{\theta \circ \phi}
    \]
    by construction. On the other hand, since $\chi: B \to A$ is differential,
    \[
    T^K_\phi \circ \chi = T^K_{\phi \circ \chi}
    \]
    by (TM2) for $T^K$.
\end{proof}

\begin{rmk}
    For $A \in \DRing$, the unit at $A$ is precisely the structure map $\eta_{A^*} = T_{\id_A}: A \to A^*$. For $B \in \DRing_\Ring$, the counit at $B$ is the evaluation map $\ev_B = \ev_B \circ \id_{B^*}: B^* \to B$.
\end{rmk}

\section{Twisting the Hurwitz Construction} \label{hurwitz_twisting_section}

In this final section, we give a concrete description of the universal Taylor morphism for an arbitrary differential ring $(K, \bdelta)$ by performing a construction analogous to that of the twisted Taylor morphism in \cite[Section 3]{SanchezTressl2020} in the case of differential $\Q$-algebras.

Let $K$ be a $\Q$-algebra, $\bdelta$ be a tuple of commuting derivations on $K$, and let $T$ denote the classical $(K, \bzero)$-Taylor morphism for $(K[[\bt]], \ddbt)$. In \cite{SanchezTressl2020}, the authors construct the twisted Taylor morphism $T^*$, which by Corollary \ref{utm_for_Q_algs_and_constants_cor} is the universal $(K, \bdelta)$-Taylor morphism for $(K[[\bt]], \bdelta + \ddbt)$, by constructing an isomorphism $\theta: (K[[\bt]], \ddbt) \to (K[[\bt]], \bdelta + \ddbt)$ of differential rings, and defining $T^*$ by setting $T^*_\phi = \theta \circ T_\phi$ for all $\phi: A \to K$. 

We will perform an analogous construction based on the Hurwitz morphism. Fix an arbitrary ring $K$, and let $H$ denote the Hurwitz morphism (considered as a $(K, \bzero)$-Taylor morphism) as defined in Example \ref{hurwitz_morphism_eg}.

\begin{defn}
    We say that two families $\bdelta$ and $\bpartial$ of commuting derivations on an arbitrary differential ring $K$ \textbf{commute} if, for any $1 \leq i \leq j \leq m$, we have that $\delta_i\circ\partial_j = \partial_j \circ\delta_i$.

    For a family of commuting derivations $\bdelta = (\delta_1,...,\delta_m)$ on $K$, we also write $\bdelta$ for the family $(\delta_1,...,\delta_m)$ of derivations on $H(K)$, where $\delta_i$ acts on the coefficients of series by $\delta_i$, i.e.
    \[
    \delta_i\left(\sum_\alpha a_\alpha \bt^\alpha\right) = \sum_\alpha \delta_i(a_\alpha) \bt^\alpha.
    \]
    By a routine computation, we can easily verify that $\bdelta$ is a commuting family of derivations on $H(K)$, and $\bdelta$ commutes with $\bd_K$ on $H(K)$.
\end{defn}

\begin{nota}
    For a family $\bdelta$ of commuting derivations on $K$, write $H_\ev^{\bdelta + \bd_K}$ for $H$ applied to the ring homomorphism $\ev: (H(K), \bdelta + \bd_K) \to (K, \bzero)$.
\end{nota}

\begin{lem} \label{H_gives_group_hom_on_derivations_lem}
    Suppose that $\bdelta$ and $\bpartial$ are commuting families of commuting derivations on $K$. Then,
    \[
    H_\ev^{\bdelta + \bpartial + \bd_K} = H_\ev^{\bdelta + \bd_K} \circ H_\ev^{\bpartial + \bd_K}.
    \]
\end{lem}

\begin{rmk}
    The above statement is directly analogous to \cite[Theorem 3.2]{SanchezTressl2020}, where the Hurwitz morphism $H$ replaces the classical Taylor morphism $T$, and the Hurwitz ring $H(K)$ replaces the standard power series ring $K[[\bt]]$.
\end{rmk}

\begin{proof}
    We verify the above equality by direct computation. Let $a = \sum_\alpha a_\alpha \bt^\alpha$. Then:
    \begin{align}
        H^{\bpartial + \bd_K}_\ev(a) &= \sum_\alpha \ev\left( (\bpartial + \bd_K)^\alpha \left( \sum_\beta a_\beta \bt^\beta \right) \right) \bt^\alpha \\
        &= \sum_\alpha \ev \left( \sum_{\gamma \leq \alpha} \binom{\alpha}{\gamma} \bpartial^\gamma \bd_K^{\alpha-\gamma} \left(\sum_\beta a_\beta \bt^\beta \right)\right) \bt^\alpha \\
        &= \sum_\alpha \ev \left( \sum_{\gamma \leq \alpha} \binom{\alpha}{\gamma} \sum_\beta \bpartial^\gamma (a_{\beta+\alpha-\gamma}) \bt^\beta \right) \bt^\alpha \\
        &= \sum_\alpha \sum_{\gamma \leq \alpha} \binom{\alpha}{\gamma} \bpartial^\gamma(a_{\alpha - \gamma}) \bt^\alpha.
    \end{align}
    Now applying $H_\ev^{\bdelta + \bd_K}$ yields (by substituting into the above, replacing $\bpartial$ with $\bdelta$):
    \begin{align}
        H_\ev^{\bdelta + \bd_K} \circ H_\ev^{\bpartial + \dd_K}(a) &= \sum_\alpha \sum_{\beta \leq \alpha} \binom{\alpha}{\beta} \bdelta^\beta \left( \sum_{\gamma \leq \alpha-\beta} \binom{\alpha-\beta}{\gamma} \bpartial^\gamma(a_{\alpha-\beta-\gamma})\right) \bt^\alpha \\
        &= \sum_\alpha \sum_{\beta+\gamma \leq \alpha} \binom{\alpha}{\beta} \binom{\alpha-\beta}{\gamma} \bdelta^\beta \bpartial^\gamma(a_{\alpha-\beta-\gamma})\bt^\alpha \\
        &= \sum_\alpha \sum_{\beta+\gamma\leq \alpha} \frac{\alpha!}{(\alpha-\beta-\gamma)!\beta!\gamma!} \bdelta^\beta\bpartial^\gamma(a_{\alpha-\beta-\gamma}) \bt^\alpha.
    \end{align}
    We now evaluate the left hand side.
    \begin{align}
        H_\ev^{\bdelta+\bpartial + \bd_K} &= \sum_\alpha \ev\left( (\bdelta+\bpartial + \bd_K)^\alpha \sum_\xi a_\xi \bt^\xi \right) \bt^\alpha \\
        &= \sum_\alpha \ev \left( \sum_{\beta+\gamma \leq \alpha} \binom{\alpha}{\beta+\gamma}\binom{\beta+\gamma}{\beta} \bdelta^\beta \bpartial^\gamma \bd_K^{\alpha-\beta-\gamma} \left(\sum_\xi a_\xi \bt^\xi\right) \right) \bt^\alpha \\
        &= \sum_\alpha \ev\left( \sum_{\beta+\gamma \leq \alpha} \binom{\alpha}{\beta+\gamma}\binom{\beta+\gamma}{\beta} \bdelta^\beta \bpartial^\gamma \left( \sum_\xi a_{\xi+\alpha-\beta-\gamma} \bt^\xi \right)\right) \bt^\alpha \\
        &= \sum_\alpha \sum_{\beta+\gamma \leq \alpha} \binom{\alpha}{\beta+\gamma} \binom{\beta+\gamma}{\beta} \bdelta^\beta \bpartial^\gamma(a_{\alpha-\beta-\gamma}) \bt^\alpha \\
        &= \sum_\alpha \sum_{\beta+\gamma\leq\alpha} \frac{\alpha!}{(\alpha-\beta-\gamma)! \beta! \gamma!} \bdelta^\beta \bpartial^\gamma(a_{\alpha-\beta-\gamma}) \bt^\alpha.
    \end{align}
    Thus the equality holds.
\end{proof}

\begin{cor} \label{hurwitz_twisting_cor}
    For a family of commuting derivations $\bdelta$ on $K$,
    \[
    H_\ev^{\bdelta + \bd_K} : (H(K), \bdelta + \bd_K) \to (K(K), \bd_K)
    \]
    is an isomorphism of differential rings with compositional inverse $H_\ev^{-\bdelta + \bd_K}$.
\end{cor}
\begin{proof}
    It is easy to see that $\bdelta$ and $-\bdelta$ commute. Then apply Lemma \ref{H_gives_group_hom_on_derivations_lem} and that $\ev: (H(K), \bd_K) \to (K, \bzero)$ is the evaluation map for $H$.
\end{proof}

This gives us the necessary ingredients with which to construct a twisting of the Hurwitz morphism.

\begin{defn}[Twisted Hurwitz Morphism]
    Let $(K, \bdelta)$ be a differential ring. For any differential ring $(A, \bpartial)$ and ring homomorphism $\phi: A \to K$, define the \textbf{twisted Hurwitz morphism} of $\phi$ as:
    \[
    H^*_\phi = H^{-\bdelta + \bd_K}_\ev \circ H_\phi: (A, \bpartial) \to (H(K), \bdelta + \bd_K).
    \]
    Explicitly computed, for any $a \in A$:
    \[
    H^*_\phi(a) = \sum_\alpha \sum_{\gamma \leq \alpha} (-1)^\gamma \binom{\alpha}{\gamma} \bdelta^\gamma (\phi(\bpartial^{\alpha-\gamma} a)) \bt^\alpha.
    \]
\end{defn}

\begin{thm} \label{twisted_hurwitz_is_utm_thm}
    For a differential ring $(K, \bdelta)$, the twisted Hurwitz morphism $H^*$ is the universal $(K, \bdelta)$-Taylor morphism for $(H(K), \bdelta + \bd_K)$. 
\end{thm}
\begin{proof}
    By Theorem \ref{universal_tm_char_by_ev_thm}, it suffices to show that $H^*$ admits an evaluation map. Let $\ev: H(K) \to K$ be the map sending a series to its constant term. We claim that this is the desired evaluation map.

    Let $(A, \bpartial)$ be a differential ring, and $\phi: A \to K$ a ring homomorphism. Then, for any $a \in A$,
    \begin{align}
        \ev\circ H^*_\phi(a) &= \ev \left( \sum_\alpha \sum_{\gamma \leq \alpha} (-1)^\gamma \binom{\alpha}{\gamma} \bdelta^\gamma (\phi(\bpartial^{\alpha-\gamma} a)) \bt^\alpha \right)  \\
        &= \phi(a)
    \end{align}
    as this is the constant term of the above series, so (EV1) holds. For (EV2), let $a = \sum_\alpha a_\alpha \bt^\alpha$. Let $H^*_\ev(a) = \sum_\alpha b_\alpha \bt^\alpha$. Using the explicit form of $H^*_\ev$, we see the following:
    \begin{align}
        b_\alpha &= \sum_{\gamma\leq \alpha} (-1)^\gamma \binom{\alpha}{\gamma} \bdelta^\gamma \ev \left((\bdelta + \bd_K)^{\alpha-\gamma} \sum_\xi a_\xi \bt^\xi \right) \\
        &= \sum_{\gamma\leq\alpha} (-1)^\gamma \binom{\alpha}{\gamma} \bdelta^\gamma \ev \left(\sum_{\beta \leq \alpha-\gamma} \binom{\alpha-\gamma}{\beta} \bdelta^{\alpha-\gamma-\beta} \bd_K^\beta \sum_\xi a_\xi \bt^\xi \right) \\
        &= \sum_{\gamma\leq\alpha} (-1)^\gamma \binom{\alpha}{\gamma} \bdelta^\gamma \ev \left( \sum_\xi \sum_{\beta\leq \alpha-\gamma} \binom{\alpha-\gamma}{\beta} \bdelta^{\alpha-\gamma-\beta}(a_{\xi+\beta}) \bt^\xi \right) \\
        &= \sum_{\gamma\leq \alpha} \sum_{\beta\leq \alpha-\gamma} (-1)^\gamma \binom{\alpha}{\gamma} \binom{\alpha-\gamma}{\beta} \bdelta^{\alpha-\beta}(a_\beta)\\
        &= \sum_{\beta+\gamma \leq \alpha} (-1)^\gamma \binom{\alpha}{\gamma} \binom{\alpha-\gamma}{\beta} \bdelta^{\alpha-\beta}(a_\beta).
    \end{align}
    Consider the coefficient $c_{\alpha,\beta}$ of $\bdelta^{\alpha-\beta}(a_\beta)$ in the above expression. Fix $\alpha, \beta$, and observe the following:
    \begin{align}
        c_{\alpha,\beta} &= \sum_{\gamma\leq \alpha-\beta} (-1)^\gamma \binom{\alpha}{\gamma} \binom{\alpha-\gamma}{\beta} \\
        &= \sum_{\gamma \leq \alpha-\beta} (-1)^\gamma \frac{\alpha!}{\gamma! \beta! (\alpha-\beta-\gamma)!} \\
        &= \frac{\alpha!}{(\alpha-\beta)!\beta!} \sum_{\gamma\leq\alpha-\beta} (-1)^\gamma \binom{\alpha-\beta}{\gamma}
    \end{align}
    which is 0 unless $\alpha=\beta$, in which case it is 1. Thus, $H^*_\ev = \id_{H(K)}$ as required.
\end{proof}

\subsection{Differentially Large Fields of Positive Characteristic}
We can reasonably ask, whether in characteristic $p \neq 0$, differentially large fields can be characterised in terms of an existential closure condition in some extension which admits a Taylor morphism, analogously to the characteristic 0 case. This turns out to be false.

\begin{lem}\label{hurwitz_ring_has_no_diff_prime_ideals_lem}
    For any differential field $(K, \bdelta)$ of characteristic $p > 0$, the (twisted) Hurwitz series ring $(H(K), \bdelta + \bd_K)$ has no differential prime ideals.
\end{lem}
\begin{proof}
    By \cite[Proposition 3.4]{Keigher1997} (applied $m$ times), and since $K$ is a field, the nilradical of $H(K)$ is
    \[
    \m \coloneqq N(H(K)) =  \left\{ \sum_\alpha a_\alpha \bt^\alpha : a_{\bar{0}} = 0\right\}
    \]
    where $\bar{0}$ denotes the $m$-tuple $(0,...,0)$. By \cite[Lemma 3.2]{Keigher1975}, every element $a \in H(K) \setminus \m$ is a unit, thus $\m$ is the unique prime (and maximal) ideal of $H(K)$. It is easy to see that $\m$ is not differential, for example, $(\delta_1 + \dd_{K, 1})(t_1) = 1 \not\in \m$. 
\end{proof}

\begin{prop}
    For a differential field $(K, \bdelta)$ of characteristic $p > 0$, no integral domain admits a $K$-Taylor morphism. In particular, if $(L, \bpartial)$ is a differential $K$-algebra which admits a $K$-Taylor morphism, then $(K, \bdelta)$ is not existentially closed in $(L, \bpartial)$.
\end{prop}
\begin{proof}
    Let $(L, \bpartial)$ be a differential $K$-algebra which admits a $K$-Taylor morphism. By Theorem \ref{twisted_hurwitz_is_utm_thm}, $(L, \bpartial)$ is a differential $(H(K), \bdelta + \bd_K)$-algebra. By Lemma \ref{hurwitz_ring_has_no_diff_prime_ideals_lem}, no differential $(H(K), \bdelta + \bd_K)$-algebra is an integral domain, as the image of $H(K)$ under any differential ring homomorphism is not a domain. The existential closure statement follows immediately.
\end{proof}

\bibliography{references}
\bibliographystyle{halpha}

\end{document}

%% file: preamble.tex
\usepackage[utf8]{inputenc}

\usepackage{amsmath}
\usepackage{amssymb}

\usepackage{bm}

\usepackage{physics}

\usepackage{xurl}

\usepackage{imakeidx}

\usepackage{hyperref}

\let\oldproof\proof
\def\proof{\oldproof\unskip}

\usepackage{etoolbox}

\makeatletter

\patchcmd\deferred@thm@head
  {\addvspace{-\parskip}}
  {}
  {}{\typeout{\string\deferred@thm@head patch failed!}}

\makeatletter 

\usepackage{mathtools}
\mathtoolsset{showonlyrefs}
\usepackage{enumitem}
\usepackage{amsthm}

\usepackage{tikz-cd}

\newtheorem{thm}[subsection]{Theorem}
\newtheorem{lem}[subsection]{Lemma}
\newtheorem{prop}[subsection]{Proposition}
\newtheorem{cor}[subsection]{Corollary}

\newtheorem{ques_num}[subsection]{Question}

\theoremstyle{remark}
\newtheorem*{rmk}{Remark}

\theoremstyle{definition}
\newtheorem{defn}[subsection]{Definition}
\newtheorem*{nota}{Notation}
\newtheorem{eg}[subsection]{Example}
\newtheorem{egs}[subsection]{Examples}

\newcommand{\id}{\mathrm{id}}

\newcommand{\iso}{\cong}

\newcommand{\bd}{\bm{\mathrm{d}}}
\newcommand{\bdelta}{\bm{\delta}}
\newcommand{\bpartial}{\bm{\partial}}
\newcommand{\bzero}{\bm{0}}

\renewcommand{\ev}{\mathrm{ev}}

\renewcommand{\leq}{\leqslant}
\renewcommand{\geq}{\geqslant}

\renewcommand{\epsilon}{\varepsilon}

\renewcommand{\phi}{\varphi}
\newcommand{\m}{\mathfrak{m}}

\newcommand{\Z}{\mathbb{Z}}
\newcommand{\N}{\mathbb{N}}
\newcommand{\Q}{\mathbb{Q}}

\newcommand{\rest}{{\upharpoonright}}

\newcommand{\ddt}{\frac{\dd}{\dd t}}
\newcommand{\bt}{\mathbf{t}}
\newcommand{\ddbt}{\frac{\dd}{\dd\mathbf{t}}}

\newcommand{\dalg}{\mathrm{dalg}}

\newcommand{\DRing}{\mathbf{DRing}}
\newcommand{\Ring}{\mathbf{Ring}}
\newcommand{\Alg}{\mathbf{\text-Alg}}
\newcommand{\DAlg}{\mathbf{\text-DAlg}}

\newcommand{\dfg}{\mathrm{dfg}}

\def\Ind#1#2{#1\setbox0=\hbox{$#1x$}\kern\wd0\hbox to 0pt{\hss$#1\mid$\hss}
\lower.9\ht0\hbox to 0pt{\hss$#1\smile$\hss}\kern\wd0}

\def\notind#1#2{#1\setbox0=\hbox{$#1x$}\kern\wd0
\hbox to 0pt{\mathchardef\nn=12854\hss$#1\nn$\kern1.4\wd0\hss}
\hbox to 0pt{\hss$#1\mid$\hss}\lower.9\ht0 \hbox to 0pt{\hss$#1\smile$\hss}\kern\wd0}

\usepackage{graphicx}


%% file: references.bib
@article{Keigher1975,
    title = {{Adjunctions and comonads in differential algebra}},
    year = {1975},
    journal = {Pacific Journal of Mathematics},
    author = {Keigher, William},
    number = {1},
    pages = {99--112},
    volume = {59},
    url = {http://msp.org/pjm/1975/59-1/p11.xhtml},
    doi = {10.2140/pjm.1975.59.99},
    issn = {0030-8730}
}

@book{MacLane1978,
    title = {{Categories for the Working Mathematician}},
    year = {1978},
    author = {Mac Lane, Saunders},
    edition = {2nd},
    pages = {115--116},
    series = {Graduate Texts in Mathematics},
    volume = {5},
    publisher = {Springer New York},
    url = {http://link.springer.com/10.1007/978-1-4757-4721-8},
    address = {New York, NY},
    isbn = {978-1-4419-3123-8},
    doi = {10.1007/978-1-4757-4721-8}
}

@book{RydeheardBurstall1988,
    title = {{Computational Category Theory}},
    year = {1988},
    author = {Rydeheard, D E and Burstall, R M},
    publisher = {Prentice Hall},
    url = {http://www.cs.man.ac.uk/~david/categories/},
    address = {New York},
    isbn = {0131627368}
}

@article{SanchezTressl2020,
    title = {{Differentially Large Fields}},
    year = {2024},
    journal = {Algebra {\&} Number Theory},
    author = {Le{\'{o}}n S{\'{a}}nchez, Omar and Tressl, Marcus},
    volume = {18},
    number = {2},
    url = {http://arxiv.org/abs/2005.00888},
    eprint = {arXiv:2005.00888}
}

@article{SanchezTressl2023,
    title={On ordinary differentially large fields}, 
      author={León Sánchez, Omar and Marcus Tressl},
      year={2024},
      journal = {Canadian Journal of Mathematics},
      pages = {1 - 22},
      eprint={arXiv:2307.12977},
      archivePrefix={arXiv},
      primaryClass={math.AG},
      url={https://arxiv.org/abs/2307.12977}, 
}

@article{Umemura1996,
   author = {Hiroshi Umemura},
   doi = {10.1017/S0027763000006012},
   issn = {00277630},
   journal = {Nagoya Mathematical Journal},
   pages = {1-58},
   publisher = {Nagoya University},
   title = {Galois theory of aigebraic and differential equations},
   volume = {144},
   year = {1996}
}

@article{Keigher1997,
    title = {{On the ring of Hurwitz series}},
    year = {1997},
    journal = {Communications in Algebra},
    author = {Keigher, William F.},
    number = {6},
    pages = {1845--1859},
    volume = {25},
    doi = {10.1080/00927879708825957},
    issn = {00927872}
}
